\documentclass[10pt]{amsart}

\usepackage{amsthm,amssymb,amsmath,amscd,epic,eepic}
\usepackage[all]{xy}
\usepackage{graphicx,subcaption}
\usepackage{amsmath} 
\usepackage{epic,eepic} 
\usepackage{yfonts}
\usepackage{paralist,enumerate}
\usepackage{hyperref}
\hypersetup{colorlinks}

\usepackage{tikz}
\usetikzlibrary {positioning}
\usetikzlibrary{decorations.markings}

\usetikzlibrary{arrows}
\tikzset{    vertex/.style={circle,draw,minimum size=1.5em},    edge/.style={->,> = latex'}}

\usepackage[skip=0pt]{caption}

\usepackage{amsmath, amsfonts, amssymb,amscd,graphics,graphicx,color,eucal,setspace} 
\usepackage{latexsym} 
\usepackage{color}

\usepackage{tikz}
\usepackage{tikz-cd}
\usepackage{tikz-3dplot}
\usetikzlibrary{shapes.geometric, calc}

\usetikzlibrary{math}

\usepackage{appendix}

\usepackage{url}
\usepackage{amscd}
\usepackage{here}
\usepackage{enumitem}

\usepackage[labelformat=simple]{subcaption}
\renewcommand\thesubfigure{(\arabic{subfigure})}

\numberwithin{equation}{section}

\theoremstyle{plain}
\newtheorem{theorem}{Theorem}[section]
\newtheorem{lemma}[theorem]{Lemma} 

\newtheorem{proposition}[theorem]{Proposition}

\theoremstyle{definition}
\newtheorem{remark}{Remark}[section]
\newtheorem{example}{Example}[section]

\def\C{\mathbb C}
\def\R{\mathbb R}

\def\Z{\mathbb Z}
\def\H{\mathbb H}
\def\O{\mathbb O}

\def\b0{\mathbf 0}

\def\b1{\mathbf 1}

\def\MJ{\mathcal J}
\def\W2{W_2}

\def\a{a}
\def\b{b}
\def\c{c}
\def\d{d}
\def\de{\delta}
\def\Sab{C_{a+b}}
\def\Scd{C_{c+d}}
\def\D{{\rm Int}\ D}

\def\a{a}
\def\b{b}
\def\c{c}
\def\d{d}
\def\de{\delta}
\def\Sab{C_{a+b}}
\def\Scd{C_{c+d}}
\def\D{{\rm Int}\ D}
\def\GM{\mathcal{G}_M}
\def\GG{K}

\DeclareMathOperator{\Int}{Int}
\DeclareMathOperator{\GL}{GL}
\DeclareMathOperator{\SU}{SU}
\DeclareMathOperator{\U}{U}
\DeclareMathOperator{\SO}{SO}
\DeclareMathOperator{\Sp}{Sp}
\DeclareMathOperator{\PU}{PU}
\DeclareMathOperator{\Aut}{Aut}
\DeclareMathOperator{\Hom}{Hom}
\DeclareMathOperator{\Real}{Re}
\DeclareMathOperator{\Imag}{Im}

\begin{document}

\tikzset{->-/.style={decoration={
  markings,
  mark=at position #1 with {\arrow{>}}},postaction={decorate}}}

\title[$6$-dimensional GKM manifolds with $4$ fixed points]{Six-dimensional GKM manifolds with four fixed points}

\author[D. Jang]{Donghoon Jang}
\address{Department of Mathematics and Institute of Mathematical Science, Pusan National University, 2, Busandaehak-ro 63beon-gil, Geumjeong-gu, Busan, 46241, Republic of Korea.}
\email{donghoonjang@pusan.ac.kr}

\author[S. Kuroki]{Shintar\^o Kuroki}
\address{Department of Applied Mathematics Faculty of Science, Okayama University of Science, 1-1 Ridai-Cho Kita-Ku Okayama-shi Okayama 700-0005, Okayama, Japan}
\email{kuroki@ous.ac.jp}

\author[M. Masuda]{Mikiya Masuda}
\address{Osaka Central Advanced Mathematical Institute, Osaka Metropolitan University, Sugimoto, Sumiyoshi-ku, Osaka, 558-8585, Japan}
\email{mikiyamsd@gmail.com}

\author[T. Sato]{Takashi Sato}
\address{Osaka Central Advanced Mathematical Institute, Osaka Metropolitan University, Sugimoto, Sumiyoshi-ku, Osaka, 558-8585, Japan}
\email{00tkshst00@gmail.com}


\begin{abstract}
In this paper, we study $6$-dimensional GKM manifolds with $4$ fixed points. We classify all possible GKM graphs, and for each type of graph we construct a manifold, proving the existence. We show that six types occur.
\begin{enumerate}
\item[(P1)] complex projective space $\C P^3$ with standard complex structure
\item[(P2)] blow up of $S^6$ at a fixed point, diffeomorphic to $\C P^3$
\item[(P3)] $\C P^3$ as the homogeneous space $\Sp(2)/(\U(1) \times \Sp(1))$ with non-standard almost complex structure
\item[(Q1)] complex quadric $Q_3$ with standard complex structure
\item[(Q2)] blow up of $S^6$ along isotropy $2$-sphere, diffeomorphic to $Q_3$
\item[(S)] $S^2 \times S^4$, obtained as equivariant gluing along orbits of two $S^6$'s
\end{enumerate}
\end{abstract}

\maketitle

\addtocontents{toc}{\protect\setcounter{tocdepth}{1}}

\section{Introduction}

A GKM manifold, named after the work of Goresky, Kottwitz, and MacPherson \cite{go-ko-ma98}, is a type of almost complex manifold endowed with a torus action, whose equivariant cohomology is described by combinatorial data, called a GKM graph \cite{GZ01}. The vertex set of the graph is the fixed point set of the action, and the edge set is the set of invariant $2$-spheres connecting two fixed points; each edge has a label that corresponds to a weight at a fixed point.

More precisely, a \textbf{GKM manifold} is a connected, compact almost complex manifold $(M,J)$ equipped with an action of a compact torus $T$ that preserves the almost complex structure $J$ and satisfies the following properties:
\begin{enumerate}
\item The fixed point set $M^T$ is finite.
\item The weights at the tangential $T$-module $T_pM$ are pairwise linearly independent for any $p\in M^T$.
\item The odd cohomology groups of $M$ vanish\footnote{Some authors include this property for a GKM manifold, while others impose this property as an additional one.}.
\end{enumerate}

To a GKM manifold $M$, we associate a GKM graph as follows; the vertex set is the fixed point set, and if $w$ is a weight at $p$ and $-w$ is a weight at $q$ such that $p,q$ are in the same $2$-sphere on which $T$ acts with weight $w$, then we draw an edge from $p$ to $q$ with label $w$, resulting in a directed labeled graph. For details, see Section~\ref{sec:graph}.

By (2), the dimension of the torus acting on a GKM manifold is at least two, so the dimension of a GKM manifold is at least four.
Four dimensional GKM manifolds are essentially classified by Orlik-Raymond \cite{OR}. Their GKM graphs are polygons with labels.
In dimension $6$, a GKM manifold with $2$ fixed points has the GKM graph of a $T^2$-action on the $6$-sphere $S^6$.
The next case is a $6$-dimensional GKM manifold with $4$ fixed points.

Six dimensional manifolds with four fixed points have been studied for $S^1$-actions.
Ahara studied $S^1$-actions on $6$-dimensional compact almost complex manifolds $M$ with $4$ fixed points and $\mathrm{Todd}(M)=1$ and $\langle c_1^3(M), [M] \rangle \neq 0$ \cite{A}. Tolman classified Hamiltonian $S^1$-actions on $6$-dimensional compact symplectic manifolds with $4$ fixed points \cite{T}. Li revisited this case \cite{Li25}. The work of \cite{J} continued this study to show that for an $S^1$-action on a $6$-dimensional compact almost complex manifold with $4$ fixed points, weights at the fixed points fall into six cases, Cases (A-F), leaving the existence of manifolds of Cases (D) and (F) unknown; see Theorem~\ref{weight-classify}.
Konstantis and Lindsay constructed a manifold of Case (D) \cite{KL}. 
For more discussion about existence, see Remark~\ref{rem:existence}.

In this paper, we study $6$-dimensional GKM manifolds with $4$ fixed points. We explain the results of this paper. Let $M$ be a $6$-dimensional GKM manifold with $4$ fixed points.
\begin{itemize}[leftmargin=1em]
\item In Section~\ref{sec:graph}, we classify all possible GKM graphs of $M$ by showing that one of the six graphs in Figure~\ref{graph-main} occurs as a GKM graph of $M$ (Theorem~\ref{gkm-classify}.)

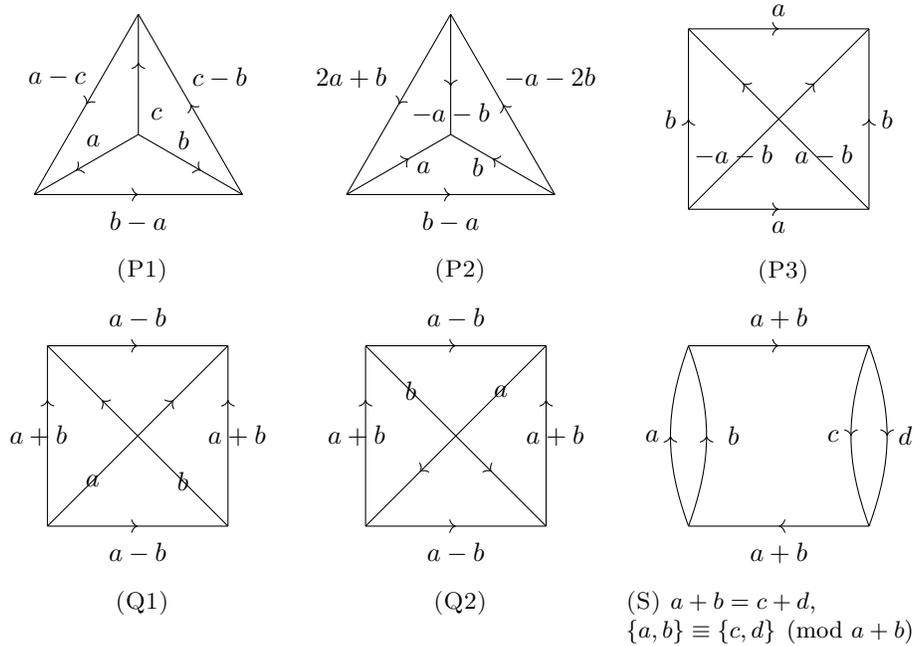
\begin{figure}[H]
\begin{subfigure}[b][4cm][s]{.32\textwidth}
\centering
\begin{tikzpicture}[state/.style ={circle, draw}]
\begin{scope}[xscale=1.2, yscale=1.2]
\draw[->-=.6] (0,0.666) to (-1.154,0);
\draw[->-=.6] (0,0.666) to (1.154,0);
\draw[->-=.6] (0,0.666) to (0,2);
\draw[->-=.5] (-1.154,0) to (1.154,0);
\draw[->-=.5] (1.154,0) to (0,2);
\draw[->-=.5] (0,2) to (-1.154,0);
\draw (0,-0.3) node{$b-a$};
\draw (-0.5,0.6) node{$a$};
\draw (0.5,0.6) node{$b$};
\draw (0.2,0.9) node{$c$};
\draw (0.9,1.3) node{$c-b$};
\draw (-0.9,1.3) node{$a-c$};
\draw (0,2.1) node{};
\end{scope}
\end{tikzpicture}
\renewcommand{\thesubfigure}{(P1)}
\caption{}
\label{graph1}
\end{subfigure} 
\begin{subfigure}[b][4cm][s]{.33\textwidth}
\centering
\begin{tikzpicture}[state/.style ={circle, draw}]
\begin{scope}[xscale=1.2, yscale=1.2]
\draw[->-=.6] (-1.154,0) to (0,0.666);
\draw[->-=.6] (1.154,0) to (0,0.666);
\draw[->-=.6] (0,2) to (0,0.666);
\draw[->-=.5] (-1.154,0) to (1.154,0);
\draw[->-=.5] (1.154,0) to (0,2);
\draw[->-=.5] (0,2) to (-1.154,0);
\draw (0,-0.3) node{$b-a$};
\draw (-0.3,0.3) node{$a$};
\draw (0.3,0.3) node{$b$};
\draw (0,0.9) node{$-a-b$};
\draw (1.1,1.3) node{$-a-2b$};
\draw (-1.1,1.3) node{$2a+b$};
\draw (0,2.1) node{};
\end{scope}
\end{tikzpicture}
\renewcommand{\thesubfigure}{(P2)}
\caption{}
\label{graph4}
\end{subfigure}
\begin{subfigure}[b][4cm][s]{.33\textwidth}
\centering
\begin{tikzpicture}[state/.style ={circle, draw}]
\begin{scope}[xscale=1.2, yscale=1.2]
\draw[->-=.5] (0,0) to (2,0);
\draw[->-=.5] (0,0) to (0,2);
\draw[->-=.7] (0,0) to (2,2);
\draw[->-=.7] (2,0) to (0,2);
\draw[->-=.5] (2,0) to (2,2);
\draw[->-=.5] (0,2) to (2,2);
\draw (1,-0.2) node{$a$};
\draw (-0.2,1) node{$b$};
\draw (0.5,0.6) node{$-a-b$};
\draw (1.5,0.6) node{$a-b$};
\draw (2.2,1) node{$b$};
\draw (1,2.2) node{$a$};
\end{scope}
\end{tikzpicture}
\renewcommand{\thesubfigure}{(P3)}
\caption{}
\label{graph3}
\end{subfigure} 
\begin{subfigure}[b][4.7cm][s]{.32\textwidth}
\centering
\begin{tikzpicture}[state/.style ={circle, draw}]
\begin{scope}[xscale=1.2, yscale=1.2]
\draw[->-=.5] (0,0) to (2,0);
\draw[->-=.7] (0,0) to (0,2);
\draw[->-=.7] (0,0) to (2,2);
\draw[->-=.7] (2,0) to (0,2);
\draw[->-=.7] (2,0) to (2,2);
\draw[->-=.5] (0,2) to (2,2);
\draw (1,-0.3) node{$a-b$};
\draw (-0.1,1) node{$a+b$};
\draw (0.5,0.5) node{$a$};
\draw (1.5,0.5) node{$b$};
\draw (2.1,1) node{$a+b$};
\draw (1,2.3) node{$a-b$};
\end{scope}
\end{tikzpicture}
\renewcommand{\thesubfigure}{(Q1)}
\caption{}
\label{graph2}
\end{subfigure} 
\begin{subfigure}[b][4.7cm][s]{.33\textwidth}
\centering
\begin{tikzpicture}[state/.style ={circle, draw}]
\begin{scope}[xscale=1.2, yscale=1.2]
\draw[->-=.5] (0,0) to (2,0);
\draw[->-=.7] (0,0) to (0,2);
\draw[->-=.7] (2,2) to (0,0);
\draw[->-=.7] (0,2) to (2,0);
\draw[->-=.7] (2,0) to (2,2);
\draw[->-=.5] (0,2) to (2,2);
\draw (1,-0.3) node{$a-b$};
\draw (-0.1,1) node{$a+b$};
\draw (0.5,1.5) node{$b$};
\draw (1.5,1.5) node{$a$};
\draw (2.1,1) node{$a+b$};
\draw (1,2.3) node{$a-b$};
\end{scope}
\end{tikzpicture}
\renewcommand{\thesubfigure}{(Q2)}
\caption{}
\label{graph5}
\end{subfigure}
\begin{subfigure}[b][4.7cm][s]{.33\textwidth}
\centering
\begin{tikzpicture}[state/.style ={circle, draw}]
\begin{scope}[xscale=1.2, yscale=1.2]
\draw[->-=.5] (2,0) to (0,0);
\draw[->-=.5] (0,0) to [bend left=20]  (0,2);
\draw[->-=.5] (0,0) to [bend right=20]  (0,2);
\draw[->-=.5] (2,2) to [bend left=20]  (2,0);
\draw[->-=.5] (2,2) to [bend right=20]  (2,0);
\draw[->-=.5] (0,2) to (2,2);
\draw (1,-0.3) node{$a+b$};
\draw (-0.4,1) node{$a$};
\draw (0.5,1) node{$b$};
\draw (1.6,1) node{$c$};
\draw (2.4,1) node{$d$};
\draw (1,2.3) node{$a+b$};
\end{scope}
\end{tikzpicture}
\renewcommand{\thesubfigure}{(S)}
\caption{$a+b=c+d$, \newline $\{a,b\} \equiv \{c,d\} \pmod{a+b}$}
\label{graph6}
\end{subfigure} 
\caption{GKM graphs of $6$-dimensional GKM manifolds with $4$ fixed points}\label{graph-main}
\end{figure}

\item In the subsequent sections (Sections~\ref{sect:p1}-\ref{sect:s}), for each possible GKM graph of $M$,
\begin{enumerate}[leftmargin=1em]
\item we compute the equivariant and ordinary cohomology and Chern classes of $M$, and
\item we establish the existence of $M$.
\end{enumerate}
\item For each GKM graph, we demonstrate the existence of $M$ as follows.
\begin{enumerate}[leftmargin=1em]
\item (Type (P1), Figure~\subref{graph1}) The complex projective space $\C P^3$ with the standard complex structure $J_{\textrm{std}}$ and with a standard linear torus action is an example of $M$ with GKM graph Figure~\subref{graph1}. 
\item (Type (P2), Figure~\subref{graph4}) Blow up $\textrm{Bl}_pS^6$ of a fixed point $p$ of $S^6$ with a torus action has GKM graph of Figure~\subref{graph4}.
\item (Type (P3), Figure~\subref{graph3}) The homogeneous space $\Sp(2)/(\U(1)\times \Sp(1))$, which is diffeomorphic to $\C P^3$, admits a non-standard almost complex structure $J_{\textrm{n-std}}$ and a $T^2$-action with GKM graph Figure~\subref{graph3}.
\item (Type (Q1), Figure~\subref{graph2}) A linear torus action on the complex quadric $Q_3$ with the standard complex structure $J_{\textrm{std}}$ has the GKM graph of Figure~\subref{graph2}.
\item (Type (Q2), Figure~\subref{graph5}) Blow up $\textrm{Bl}_{S^2}S^6$ of $S^6$ along an isotropy $2$-sphere $S^2$ has GKM graph of Figure~\subref{graph5}. 
\item (Type (S), Figure~\subref{graph6})
We demonstrate the existence of manifold $M$ with GKM graph Figure~\subref{graph6}, as equivariant gluing of two $S^6$'s along orbits in isotropy $2$-spheres.
\end{enumerate}
\item In Section~\ref{sec:auto}, we provide background for the automorphism group of a GKM manifold, and discuss (the connected maximal compact closed subgroup of) the automorphism group of $M$ for each type.
\item In Appendix~\ref{sec:octo}, we discuss almost complex structures on $S^6$ and $S^2 \times S^4$ as subsets of octonions.
\item In Appendix~\ref{sec:deform}, we discuss deformation of an almost complex structure of a manifold endowed with an action of a compact Lie group $G$. This enables us to blow up an isotropy submanifold of an almost complex $G$-manifold.
\item In Appendix~\ref{sec:faithful}, we discuss weights of faithful $T^n$-modules.
\end{itemize}

\begin{remark} \label{rem:sim} We note the following:
\begin{enumerate}[leftmargin=2em]
\item If a GKM manifold of Type (P2) or (P3) is simply connected, such as the manifolds $\textrm{Bl}_pS^6$ of Type (P2) and $(\C P^3,J_{\textrm{n-std}})$ of Type (P3), then by the classification of simply connected $6$-manifolds by Wall-Jupp-\^Zubr \cite{Wa, Ju, Zu}, the manifold is diffeomorphic to the complex projective space $\C P^3$, but has an almost complex structure different from the standard complex structure on $\C P^3$.
\item Similarly, if a GKM manifold of Type (Q2) is simply connected, such as the manifold $\textrm{Bl}_{S^2}S^6$, then the manifold is diffeomorphic to the complex quadric $Q_3$, but has an almost complex structure different from the standard complex structure on $Q_3$.
\item Also, if a GKM manifold of Type (S) is simply connected, such as the manifold we construct in Subsection~\ref{sec:s}, then the manifold is diffeomorphic to $S^2 \times S^4$.
\end{enumerate}
\end{remark}

\section*{Acknowledgements}

D.~Jang was supported by the National Research Foundation of Korea(NRF) grant funded by the Korea government(MSIT) (2021R1C1C1004158).
S.~Kuroki was supported by JSPS Grant-in-Aid for Scientific Research 21K03262.
M.~Masuda was supported in part by JSPS Grant-in-Aid for Scientific Research 25K07007 and the HSE University Basic Research Program.
T.~Sato was supported in part by JSPS Grant-in-Aid for Scientific Research 25K00205.
This work was partly supported by MEXT Promotion of Distinctive Joint Research Center Program JPMXP0723833165.

\section{Classification of GKM graphs} \label{sec:graph}

Let $\Gamma$ be a finite regular graph. Let $V(\Gamma)$ and $E(\Gamma)$ be the vertex set and the oriented edge set of $\Gamma$, where if $e \in E(\Gamma)$, then $e$ with the opposite orientation, denoted $\bar{e}$, is also in $E(\Gamma)$.

Let $T$ be a compact torus and let $BT$ be the classifying space of $T$.
A natural isomorphism between $H^2(BT)$ and the group $\Hom(T,S^1)$ of homomorphisms from $T$ to the unit circle $S^1$ of $\C$ enables us to consider an element $u$ of $H^2(BT)$ as a homomorphism $\chi^u$ from $T$ to $S^1$, and vice versa.
Precisely speaking, an element of $H^2(BT)$ is the weight of the corresponding element of $\Hom(T,S^1)$.

For an edge $e$ of $\Gamma$, let $i(e)$ and $t(e)$ denote the initial vertex and the terminal vertex of $e$.
We call an assignment $\alpha\colon E(\Gamma)\to H^2(BT)$ an \textbf{axial function}, if it satisfies the following:
\begin{enumerate}
\item $\alpha(\bar{e})=-\alpha(e)$,
\item elements in $\Lambda_p:=\{\alpha(e)\mid i(e)=p\}$ are pairwise linearly independent for any $p\in V(\Gamma)$,
\item $\Lambda_{i(e)}\equiv \Lambda_{t(e)}\pmod{\alpha(e)}$ for any $e\in E(\Gamma)$. That is, there is a bijection between $\Lambda_{i(e)}$ and $\Lambda_{t(e)}$ such that the congruence relation holds for each corresponding pair.
\end{enumerate}

A \textbf{GKM graph} is a pair $(\Gamma,\alpha)$, where $\Gamma$ is a finite regular graph with oriented edge set, and $\alpha$ is an axial function.

By definition, the equivariant graph cohomology of a GKM graph $(\Gamma,\alpha)$ is
\begin{equation} \label{eq:graph_equivariant_cohomology}
H^*_T(\Gamma,\alpha):=\left\{ \xi\in \mathrm{Map}(V(\Gamma),H^*(BT)) \ \left|\ \begin{array}{c}\xi(i(e))\equiv \xi(t(e))\mod {\alpha(e)} \\ \text{ for any } e\in E(\Gamma)\end{array} \right.\right\}.
\end{equation}
Here, $\mathrm{Map}(V(\Gamma),H^*(BT))$ denotes the set of maps from $V(\Gamma)$ to $H^*(BT)$.

Let $M$ be a GKM manifold. We associate a GKM graph $(\Gamma_M,\alpha_M)$ to $M$ as follows.
\begin{enumerate}
\item The vertex set $V(\Gamma_M)$ is equal to the fixed point set $M^T$ of $M$.
\item For a fixed point $p$, let $u \in H^2(BT) \cong \Hom(T,S^1)$ be a weight in the complex $T$-module $T_pM$. Since the weights at $p$ are pairwise linearly independent, the fixed component of $M^{\ker \chi^u}$ containing $p$ is a $2$-sphere $S_u$.
The $T$-action on $S_u$ contains another fixed point, say $q$. We think of $S_u$ as an edge joining $p$ and $q$, and if $e$ is the edge with orientation from $p$ to $q$, then we assign $u$ to $e$ and $-u$ to $\bar{e}$.
\end{enumerate}

\begin{remark} \label{rem:coho}
The equivariant cohomology of $M$ is isomorphic to the equivariant graph cohomology of its GKM graph $(\Gamma_M,\alpha_M)$, see \cite{GZ01}. Therefore, in Sections~\ref{sect:p1}-\ref{sect:s}, for the equivariant cohomology of a $6$-dimensional GKM manifold with $4$ fixed points for each GKM graph, we compute the equivariant graph cohomology of its GKM graph.
\end{remark}

\begin{remark} \label{rem:J}
Throughout this paper, by an action of a Lie group on an almost complex manifold $(M,J)$, we always mean an action that preserves the almost complex structure $J$.
\end{remark}


We recall the main result of \cite{J}.

\begin{theorem} \cite{J} \label{weight-classify} 
Let the circle act effectively on a $6$-dimensional compact almost complex manifold $M$ with $4$ fixed points. Then one of the following holds:
\begin{enumerate}
\item[(A)] $\mathrm{Todd}(M)=1$ and the multisets of the weights at the fixed points are $\{a,b,c\}$, $\{-a,b-a,c-a\}$, $\{-b,a-b,c-b\}$, $\{-c,a-c,b-c\}$ for some mutually distinct positive integers $a,b,c$ such that $\gcd(a,b,c)=1$.
\item[(B)] $\mathrm{Todd}(M)=1$ and the multisets of the weights at the fixed points are $\{a,a+b,a+2b\}$, $\{-a,b,a+2b\}$, $\{-a-2b,-b,a\}$, $\{-a-2b,-a-b,-a\}$ for some positive integers $a,b$ such that $\gcd(a,b)=1$.
\item[(C)] $\mathrm{Todd}(M)=1$ and the multisets of the weights at the fixed points are $\{1,2,3\}$, $\{-1,1,a\}$, $\{-1,-a,1\}$, $\{-1,-2,-3\}$ for some positive integer $a$. 
\item[(D)] $\mathrm{Todd}(M)=0$ and the multisets of the weights at the fixed points are $\{-a-b,a,b\}$, $\{-c-d,c,d\}$, $\{-a,-b,a+b\}$, $\{-c,-d,c+d\}$ for some positive integers $a,b,c,d$ such that $\gcd(a,b)=\gcd(c,d)=1$.
\item[(E)] $\mathrm{Todd}(M)=0$ and the multisets of the weights at the fixed points are $\{-3a-b,a,b\}$, $\{-2a-b,3a+b,3a+2b\}$, $\{-a,-a-b,2a+b\}$, $\{-b,-3a-2b,a+b\}$ for some positive integers $a,b$ such that $\gcd(a,b)=1$, by reversing the circle action if necessary.
\item[(F)] $\mathrm{Todd}(M)=0$ and the multisets of the weights at the fixed points are $\{-a-b,2a+b,b\}$, $\{-2a-b,a,b\}$, $\{-b,-2a-b,a+b\}$, $\{-a,-b,2a+b\}$ for some positive integers $a,b$ such that $\gcd(a,b)=1$.
\end{enumerate} \end{theorem}

\begin{remark} \label{rem:existence}
We discuss the existence of a manifold of Theorem~\ref{weight-classify}.
\begin{enumerate}
\item A standard linear $S^1$-action on $\C P^3$ with the standard complex structure provides an example of Case (A).
\item A standard linear $S^1$-action on the complex quadric $Q_3$ with the standard complex structure provides an example of Case (B).
\item Consider Case (C). Case (C) with $a=2$ ($a=3$) belongs to Case (A) (respectively, Case (B)). It is known that Fano 3-folds $V_5$ and $V_{22}$ are the examples of Case (C) for $a=4$ and $a=5$, respectively. For other values of $a$, the existence of a manifold is not known.
\item In \cite{KL}, Konstantis and Lindsay used equivariant connected sum along free orbits of two $S^6$'s, to construct an example of Case (D); their example is diffeomorphic to $S^2 \times S^4$.
\item The examples $\textrm{Bl}_pS^6$ in Subsection~\ref{sec:p3} and $\textrm{Bl}_{S^2}S^6$ in Subsection~\ref{sec:q2} below are examples of Case (E) and (F) when we restrict the $T^2$-actions to generic subcircles.
\end{enumerate}
\end{remark}

We determine the Chern numbers of the manifolds in Theorem~\ref{weight-classify}. For this, we recall the following theorem.

\begin{theorem} [ABBV localization theorem] \cite{AB, BV} \label{localization} 
Let the circle act on a compact oriented manifold $M$. For any $\alpha \in H_{S^1}^*(M;\mathbb{Q})$,
\begin{center}
$\displaystyle \int_M \alpha = \sum_{F \subset M^{S^1}} \int_F \frac{\alpha|_F}{e_{S^1}(N_F)}$.
\end{center}
Here, the sum is taken over all fixed components, and $e_{S^1}(N_F)$ is the equivariant Euler class of the normal bundle of $F$.
\end{theorem}

\begin{proposition} \label{chern}
Let the circle act on a $6$-dimensional compact almost complex manifold $M$ with 4 fixed points. In each of the cases of Theorem~\ref{weight-classify}, the manifold $M$ has the following Chern numbers.
\begin{enumerate}
\item[(A)] $\int_M c_1c_2=24$, $\int_M c_1^3=64$
\item[(B)] $\int_M c_1c_2=24$, $\int_M c_1^3=54$
\item[(C)] $\int_M c_1c_2=24$, $\int_M c_1^3=72-2a^2$
\item[(D)] $\int_M c_1c_2=0$, $\int_M c_1^3=0$
\item[(E)] $\int_M c_1c_2=0$, $\int_M c_1^3=-8$
\item[(F)] $\int_M c_1c_2=0$, $\int_M c_1^3=-2$
\end{enumerate}
\end{proposition}

\begin{proof}
Since the Todd genus of $M$ is equal to $\int_M \frac{c_1c_2}{24}$ (for instance, see \cite[p. 20]{HBJ}), the formula for $\int_M c_1c_2$ follows. By Theorem~\ref{localization},
\begin{center}
$\displaystyle \int_M c_1^3=\sum_{p \in M^{S^1}} \frac{(\sum_i w_{p,i})^3}{\prod_j w_{p,j}},$
\end{center}
where $w_{p,i}$ are the weights at $p$.

For instance, in Case (C) of Theorem~\ref{weight-classify},
\begin{center}
$\int_M c_1^3=\frac{(1+2+3)^3}{1 \cdot 2 \cdot 3}+\frac{((-1)+1+a)^3}{(-1) \cdot 1 \cdot a}+\frac{((-1)+(-a)+1)^3}{(-1) \cdot (-a) \cdot 1}+\frac{((-1)+(-2)+(-3))^3}{(-1) \cdot (-2) \cdot (-3)}$

$=72-2a^2.$
\end{center}
Similarly, the remaining cases follow.
\end{proof}

With this, the following proposition holds.

\begin{proposition} \label{no-extend}
Let $M$ be a $6$-dimensional compact almost complex manifold $(M,J)$ endowed with an $S^1$-action. Suppose that the action has $4$ fixed points and the weights at the fixed points are $\{1,2,3\}$, $\{-1,1,a\}$, $\{-1,-a,1\}$, $\{-1,-2,-3\}$ for some positive integer $a$. That is, assume that the weights at the fixed points fall into Case (C) of Theorem~\ref{weight-classify}. If $a=1$ or $a \geq 4$, then the $S^1$-action does not extend to an effective $T^2$-action 
(preserving $J$, as remarked in Remark~\ref{rem:J}).
\end{proposition}

\begin{proof}
Assume for a contradiction that the $S^1$-action extends to an effective $T^2$-action. By Proposition~\ref{chern}, $\int_M c_1c_2=24$ and $\int_M c_1^3=72-2a^2$. Since the Chern numbers are the invariants of $M$ and only Case (C) of Theorem~\ref{weight-classify} has these Chern numbers for $a=1$ or $a \geq 4$, the weights at the fixed points must be
\begin{center}
$\{1,2,3\}$, $\{-1,1,a\}$, $\{-1,-a,1\}$, $\{-1,-2,-3\}$
\end{center}
for an action of a generic $S^1$ inside $T^2$. However, we can always find a subcircle of $T^2$ whose weights at the fixed points are different from the above.
\end{proof}

We classify GKM graphs of $6$-dimensional GKM manifolds with $4$ fixed points as follows.

\begin{theorem} \label{gkm-classify}
Let $M$ be a $6$-dimensional GKM manifold with an effective $T^2$-action and with $4$ fixed points. 
Then one of figures in Figure~\ref{graph-main} occurs as a GKM graph of $M$.
In each case, the weights at the fixed points are as follows, for some $a,b,c \in H^2(BT^2)$.

\begin{enumerate}
\item[(P1)] (Figure~\subref{graph1}) $\{a,b,c\}$, $\{-a,b-a,c-a\}$, $\{-b,a-b,c-b\}$, $\{-c,a-c,b-c\}$
\item[(P2)] (Figure~\subref{graph4}) $\{-a,-b,a+b\}$, $\{a,b-a,-2a-b\}$, $\{b,a-b,-a-2b\}$, $\{-a-b,2a+b,a+2b\}$
\item[(P3)] 
(Figure~\subref{graph3}) $\{a,b,-a-b\}$, $\{-a,b,a-b\}$, $\{-a,-b,a+b\}$, $\{a,-b,b-a\}$
\item[(Q1)] (Figure~\subref{graph2}) 
$\{a,a-b,a+b\}$, $\{b,b-a,a+b\}$, $\{-a,b-a,-a-b\}$, $\{-b,a-b,-a-b\}$

\item[(Q2)] (Figure~\subref{graph5}) 
$\{-a,a-b,a+b\}$, $\{-b,b-a,a+b\}$, $\{a,b-a,-a-b\}$, $\{b,a-b,-a-b\}$

\item[(S)] (Figure~\subref{graph6}) $\{a,b,-a-b\}$, $\{-a+k(a+b),-b-k(a+b),a+b\}$, $\{-a,-b,a+b\}$, $\{a-k(a+b),b+k(a+b),-a-b\}$ for some integer $k$.
\end{enumerate}
For Type (P1) $a,b,c$ generate $H^2(BT^2)$ and for other types $a,b$ form a basis of $H^2(BT^2)$.
\end{theorem}

\begin{proof}
For an action of a generic $S^1$ of $T^2$, the weights of this $S^1$-action at the fixed points belong to one of the cases in Theorem~\ref{weight-classify} if we identify $H^2(BS^1)$ with $\mathbb{Z}$. By Proposition~\ref{no-extend}, Case (C) of Theorem~\ref{weight-classify} with $a=1$ or $a \geq 4$ does not occur as such a case. In addition, Case (C) with $a=2$ and Case (C) with $a=3$ of Theorem~\ref{weight-classify} each belong to Case (A) and Case (B), respectively. 
Therefore, the weights at the fixed points of an action of a generic $S^1$ of $T^2$ fall into one of Cases (A), (B), (D), (E), (F) of Theorem~\ref{weight-classify}.

By Proposition~\ref{chern}, Cases (A), (B), (D), (E), (F) of Theorem~\ref{weight-classify} have mutually distinct Chern numbers. This implies that if the weights of some $S^1$-action on $M$ fall into one case of Theorem~\ref{weight-classify}, then the weights of all generic $S^1$-actions are as in the case. This implies that the $T^2$-weights are also of the form in the case.

Let $\Gamma$ denote the GKM graph of $M$. The equivariant cohomology of $M$ is the graph cohomology of $\Gamma$ \cite{GZ01}. Since $M$ is connected, this implies that $\Gamma$ is connected. It follows that $\Gamma$ is one of graphs in Figure~\ref{fig1}.

\begin{figure}[H]
\begin{subfigure}[b][2.8cm][s]{.4\textwidth}
\centering
\begin{tikzpicture}[state/.style ={circle, draw}]
\draw[->-=.5] (0,0) to (2,0);
\draw[->-=.5] (0,0) to (0,2);
\draw[->-=.7] (0,0) to (2,2);
\draw[->-=.7] (2,0) to (0,2);
\draw[->-=.5] (2,0) to (2,2);
\draw[->-=.5] (0,2) to (2,2);
\end{tikzpicture}
\caption{} \label{fig1a}
\end{subfigure} 
\begin{subfigure}[b][2.8cm][s]{.4\textwidth}
\centering
\begin{tikzpicture}[state/.style ={circle, draw}]
\draw[->-=.5] (2,0) to (0,0);
\draw[->-=.5] (0,0) to [bend left=20]  (0,2);
\draw[->-=.5] (0,0) to [bend right=20]  (0,2);
\draw[->-=.5] (2,2) to [bend left=20]  (2,0);
\draw[->-=.5] (2,2) to [bend right=20]  (2,0);
\draw[->-=.5] (0,2) to (2,2);
\end{tikzpicture}
\caption{} \label{fig1b}
\end{subfigure} 
\caption{}\label{fig1}
\end{figure}

(Case (A)) Suppose that the $T^2$-weights at the fixed points are
\begin{center}
$\{a,b,c\}$, $\{-a,b-a,c-a\}$, $\{-b,a-b,c-b\}$, $\{-c,a-c,b-c\}$
\end{center}
for some $a,b,c \in H^2(BT^2)$ such that any two of them are linearly independent.
Since the action is effective, $a,b,c$ generate $H^2(BT^2)$, see Appendix~\ref{sec:faithful}.
Since $a,b,c$ are pairwise linearly independent, considering congruence relations (that is, if there is an edge from $p$ to $q$ with label $w$, then the weights at $p$ and the weights at $q$ are congruent modulo $w$), one can check that $\Gamma$ is a complete graph (that is, $\Gamma$ is Figure~\ref{fig1a}), and furthermore $\Gamma$ is Figure~\subref{graph1}. This is Type (P1) of this theorem.

(Case (B)) Suppose that the $T^2$-weights at the fixed points are
\begin{center}
$\{a,a+b,a+2b\}$, $\{-a,b,a+2b\}$, $\{a,-b,-a-2b\}$, $\{-a,-a-b,-a-2b\}$
\end{center}
for some $a,b \in H^2(BT^2)$.
Since the action is effective, $a,b$ form a basis of $H^2(BT^2)$, see Appendix~\ref{sec:faithful}.
Since $a$ and $b$ are linearly independent, considering congruence relations, one can show that $\Gamma$ is Figure~\ref{fig-q1}. 

\begin{figure}[H]
\begin{tikzpicture}[state/.style ={circle, draw}]
\begin{scope}[xscale=1.2, yscale=1.2]
\draw[->-=.5] (0,0) to (2,0);
\draw[->-=.7] (0,0) to (0,2);
\draw[->-=.7] (0,0) to (2,2);
\draw[->-=.7] (2,0) to (0,2);
\draw[->-=.7] (2,0) to (2,2);
\draw[->-=.5] (0,2) to (2,2);
\draw (1,-0.3) node{$a$};
\draw (-0.2,1) node{$a+2b$};
\draw (0.5,0.5) node{$a+b$};
\draw (1.6,0.6) node{$b$};
\draw (2.2,1) node{$a+2b$};
\draw (1,2.3) node{$a$};
\end{scope}
\end{tikzpicture}
\caption{} \label{fig-q1}
\end{figure}
Replacing $a$ with $a-b$ in Figure~\ref{fig-q1}, we get Figure~\subref{graph2}. This is Type (Q1) of this theorem.

(Case (D)) Suppose that the $T^2$-weights at the fixed points $p_1,p_2,p_3,p_4$ are of the form of Case (D) of Theorem~\ref{weight-classify}. That is, suppose that they are
\begin{center}
$\{x,a,b\}$, $\{-a,-b,-x\}$, $\{y,c,d\}$, $\{-c,-d,-y\},$
\end{center}
respectively, for some $a,b,x,c,d,y \in H^2(BT^2)$ such that $a+b+x=0$, $c+d+y=0$.
Since the action is effective, $a,b$ ($c,d$, respectively) form a basis of $H^2(BT^2)$.

Assume that $\Gamma$ is a complete graph, that is, $\Gamma$ is Figure~\ref{fig1a}. 
Then there is one edge from $p_1$ to $p_2$, one edge from $p_1$ to $p_3$, and one edge from $p_1$ to $p_4$, with labels, say $x$, $a$, and $b$, up to permuting $x$, $a$, and $b$; see Figure~\ref{fig-21}.
Since $p_2$ has weight $-b$, there must be an edge with label $b$ to $p_2$ from $p_3$ or $p_4$. 
If it is $p_4$, then $p_4$ has weights $-b$ and $b$, and they are not linearly independent.
Thus, there is an edge with label $b$ from $p_3$ to $p_2$.
Then $-a$ and $b$ are weights at $p_3$. Since the sum of the weights at $p_3$ is zero, the remaining weight at $p_3$ is $a-b$, which is the label of the edge from $p_3$ to $p_4$.
Therefore, $\Gamma$ is Figure~\subref{graph3}. This is Type (P3) of this theorem.

\begin{figure}[H]
\centering
\begin{tikzpicture}[state/.style ={circle, draw}]
\draw[->-=.5] (0,0) to (2,0);
\draw[->-=.5] (0,0) to (0,2);
\draw[->-=.7] (0,0) to (2,2);
\draw[->-=.7] (2,0) to (0,2);
\draw[->-=.5] (2,0) to (2,2);
\draw[->-=.5] (0,2) to (2,2);
\draw (1,-0.2) node{$a$};
\draw (-0.2,1) node{$b$};
\draw (0.5,0.6) node{$x$};
\draw (0,-0.3) node{$p_1$};
\draw (0,2.3) node{$p_4$};
\draw (2,-0.3) node{$p_3$};
\draw (2,2.3) node{$p_2$};
\end{tikzpicture}
\caption{} \label{fig-21}
\end{figure}

Assume that $\Gamma$ is not complete, that is, $\Gamma$ is Figure~\ref{fig1b}. Permuting $a,b,x$, we may assume without loss of generality that there are two edges from $p_1$ to $q_1$ with labels $a$ and $b$ and one edge from $q_2$ to $p_1$ with label $-x$ for some $\{q_1,q_2\} \subset \{p_2,p_3,p_4\}$; see Figure~\ref{fig-22a}. 
\begin{figure}[H]
\centering
\begin{subfigure}[b][3.3cm][s]{.49\textwidth}
\centering
\begin{tikzpicture}[state/.style ={circle, draw}]
\draw[->-=.5] (2,0) to (0,0);
\draw[->-=.5] (0,0) to [bend left=20]  (0,2);
\draw[->-=.5] (0,0) to [bend right=20]  (0,2);
\draw (1,-0.2) node{$-x$};
\draw (-0.4,1) node{$a$};
\draw (0.5,1) node{$b$};
\draw (0,-0.2) node{$p_1$};
\draw (0,2.2) node{$q_1$};
\draw (2,-0.2) node{$q_2$};
\draw (1,-0.5) node{$=a+b$};
\end{tikzpicture}
\caption{} \label{fig-22a}
\end{subfigure} 
\begin{subfigure}[b][3.3cm][s]{.49\textwidth}
\centering
\begin{tikzpicture}[state/.style ={circle, draw}]
\draw[->-=.5] (2,0) to (0,0);
\draw[->-=.5] (0,0) to [bend left=20]  (0,2);
\draw[->-=.5] (0,0) to [bend right=20]  (0,2);
\draw[->-=.5] (2,2) to [bend left=20]  (2,0);
\draw[->-=.5] (2,2) to [bend right=20]  (2,0);
\draw[->-=.5] (0,2) to (2,2);
\draw (1,-0.2) node{$-x$};
\draw (-0.4,1) node{$a$};
\draw (0.5,1) node{$b$};
\draw (1,2.2) node{$a+b$};
\draw (0,-0.2) node{$p_1$};
\draw (0,2.2) node{$q_1$};
\draw (2,-0.2) node{$q_2$};
\draw (2,2.2) node{$q_3$};
\draw (1,-0.5) node{$=a+b$};
\end{tikzpicture}
\caption{} \label{fig-22b}
\end{subfigure} \hspace{1em}
\caption{}\label{}
\end{figure}

Then $-a,-b$ are two weights at $q_1$; since the sum of the weights at $q_1$ is $0$, the remaining weight at $q_1$ is $a+b$.
Let the remaining fixed point be $q_3$. Then there is one edge from $q_1$ to $q_3$ with label $a+b$ and two edges from $q_3$ to $q_2$; see Figure~\ref{fig-22b}.
Since there is an edge from $q_2$ to $p_1$ with label $-x=a+b$, the weights $\{a,b,x\}$ at $p_1$ and the weights at $q_2$ are congruent modulo $a+b$.
This implies that the weights at $q_2$ are $\{a+l(a+b),b+l'(a+b),a+b\}$ for some integers $l,l'$.
Since the sum of the weights 
at $q_2$ is zero, it follows that $2+l+l'=0$. Letting $l'=k-1$, the weights at $q_2$ are $\{-a+k(a+b), -b-k(a+b), a+b\}$.
Therefore, $\Gamma$ is Figure~\subref{graph6}. This is Type (S) of this theorem.

(Case (E)) Suppose that the $T^2$-weights at the fixed points are
\begin{center}
$\{-3a-b,a,b\}$, $\{-2a-b,3a+b,3a+2b\}$, $\{-a,-a-b,2a+b\}$, $\{-b,-3a-2b,a+b\}$
\end{center}
for some $a,b \in H^2(BT^2)$; then $a,b$ form a basis.
Then one can see that $\Gamma$ is Figure~\ref{fig-p2}.
\begin{figure}[H]
\centering
\begin{tikzpicture}[state/.style ={circle, draw}]
\begin{scope}[xscale=1.2, yscale=1.2]
\draw[->-=.6] (-1.154,0) to (0,0.666);
\draw[->-=.6] (1.154,0) to (0,0.666);
\draw[->-=.6] (0,2) to (0,0.666);
\draw[->-=.5] (-1.154,0) to (1.154,0);
\draw[->-=.5] (1.154,0) to (0,2);
\draw[->-=.5] (0,2) to (-1.154,0);
\draw (0,-0.3) node{$b$};
\draw (-0.3,0.3) node{$a$};
\draw (0.3,0.2) node{$a+b$};
\draw (0,0.9) node{$-2a-b$};
\draw (1.1,1.3) node{$-3a-2b$};
\draw (-1.1,1.3) node{$3a+b$};
\draw (0,2.1) node{};
\end{scope}
\end{tikzpicture}
\caption{} \label{fig-p2}
\end{figure}
Replacing $b$ with $b-a$ in Figure~\ref{fig-p2}, we get Figure~\subref{graph4}.
This is Type (P2) of this theorem.

(Case (F)) Suppose that the $T^2$-weights at the fixed points are
\begin{center}
$\{-a-b,2a+b,b\}$, $\{-2a-b,a,b\}$, $\{-b,-2a-b,a+b\}$, $\{-a,-b,2a+b\}$
\end{center}
for some $a,b \in H^2(BT^2)$; then $a,b$ form a basis.

Suppose that $\Gamma$ is not complete, that is, $\Gamma$ is Figure~\ref{fig1b}. The sum of weights at each fixed point is $a+b$, $-a$, $-a-b$, $a$, respectively. 
Then for two vertices connected by two edges (multiple edges), the difference of the sums of the weights at these points 
must be
 divisible by the labels of the edges.
However, one can easily see that this is not the case. Therefore, $\Gamma$ is a complete graph. 
It follows that $\Gamma$ is Figure~\ref{fi-q2}.
\begin{figure}[H]
\centering
\begin{tikzpicture}[state/.style ={circle, draw}]
\begin{scope}[xscale=1.2, yscale=1.2]
\draw[->-=.5] (0,0) to (2,0);
\draw[->-=.7] (0,0) to (0,2);
\draw[->-=.7] (2,2) to (0,0);
\draw[->-=.7] (0,2) to (2,0);
\draw[->-=.7] (2,0) to (2,2);
\draw[->-=.5] (0,2) to (2,2);
\draw (1,-0.3) node{$b$};
\draw (-0.3,1) node{$2a+b$};
\draw (0.5,1.5) node{$a$};
\draw (1.5,1.5) node{$a+b$};
\draw (2.3,1) node{$2a+b$};
\draw (1,2.3) node{$b$};
\end{scope}
\end{tikzpicture}
\caption{} \label{fi-q2}
\end{figure}
Replacing $a$ with $b$ and $b$ with $a-b$ in Figure~\ref{fi-q2}, we get Figure~\subref{graph5}.
This is Type (Q2) of this theorem.
\end{proof}

For the manifold of Type (P1) of Theorem~\ref{gkm-classify}, the action of $T^2$ may extend to an effective $T^3$-action. For other types, by an argument analogous to Proposition~\ref{no-extend}, the $T^2$-action does not extend to a $T^3$-action as is shown below.

\begin{proposition} \label{pro-extend}
Each of the $T^2$-actions of Types (P2,P3,Q1,Q2,S) of Theorem~\ref{gkm-classify} does not extend to an effective $T^3$-action. Consequently, each of Types (P2,P3,Q1,Q2,S) of Theorem~\ref{gkm-classify} does not admit an effective $T^3$-action.
\end{proposition}

\begin{proof}
This follows from Theorem~\ref{weight-classify} because there are two parameters $a$ and $b$ in Cases (B), (D), (E), and (F), and each of these cases has unique Chern numbers.
\end{proof}

\begin{remark}
\cite[Theorem 4.2]{M} and \cite[Theorem 1.7]{J2} proved that if $T^n$ acts effectively on a $2n$-dimensional compact connected almost complex manifold $M$ with fixed points (that is, $M$ is an almost complex torus manifold), then the Todd genus of $M$ is positive. This gives another proof that each of $T^2$-actions of Types (P2,P3,Q2,S) of Theorem~\ref{gkm-classify} does not extend to an effective $T^3$-action, since manifolds in Types (P2,P3,Q2,S) have vanishing Todd genus. 
Moreover, \cite[Theorem 4.6]{M} says that if $M$ is an almost complex torus manifold and $H^*(M) \cong H^*(\C P^n)$ as groups, then $H^*(M) \cong H^*(\C P^n)$ as rings, so the $T^2$-action of Type (Q1) also does not extend to an effective $T^3$-action.
\end{remark}

\section{Type (P1)} \label{sect:p1}

\begin{remark}
From Section~\ref{sect:p1} to Section~\ref{sect:s}, we set $T=T^2$.
\end{remark}

In this section, let $M$ denote a GKM manifold of Type (P1), that is, Figure~\subref{graph1} is the GKM graph of $M$.

\subsection{Equivariant cohomology and Chern classes} \label{sec:p1-coho}

Let $\{\a,\b,\c\}$ be elements of $H^2(BT)$ that are non-zero and mutually distinct, and generate $H^2(BT)$. The GKM graph (on the left side below) can be obtained by replacing $(\a,\b,\c)$ with $(-\a,-\b,-\c)$ in Figure~\subref{graph1}.

\begin{figure}[H]
\begin{subfigure}[b][4.3cm][s]{.4\textwidth}
\centering
\begin{tikzpicture}[state/.style ={circle, draw}]
\begin{scope}[xscale=1.3, yscale=1.3]
\draw[->-=.6] (-1.154,0) to (0,0.666);
\draw[->-=.6] (1.154,0) to (0,0.666);
\draw[->-=.6] (0,2) to (0,0.666);
\draw[->-=.5] (-1.154,0) to (1.154,0);
\draw[->-=.5] (1.154,0) to (0,2);
\draw[->-=.5] (0,2) to (-1.154,0);
\draw (0,-0.3) node{$a-b$};
\draw (-0.5,0.6) node{$a$};
\draw (0.5,0.6) node{$b$};
\draw (0.2,0.9) node{$c$};
\draw (0.9,1.3) node{$b-c$};
\draw (-0.9,1.3) node{$c-a$};
\draw (0,2.3) node{};
\end{scope}
\end{tikzpicture}
\caption{}
\end{subfigure}
\begin{subfigure}[b][4.3cm][s]{.4\textwidth}
\centering
\begin{tikzpicture}[state/.style ={circle, draw}]
\begin{scope}[xscale=1.3, yscale=1.3]
\draw[->-=.6] (-1.154,0) to (0,0.666);
\draw[->-=.6] (1.154,0) to (0,0.666);
\draw[->-=.6] (0,2) to (0,0.666);
\draw[->-=.5] (-1.154,0) to (1.154,0);
\draw[->-=.5] (1.154,0) to (0,2);
\draw[->-=.5] (0,2) to (-1.154,0);
\draw (0.2,0.9) node{$0$};
\draw (-1.154,-0.3) node{$a$};
\draw (1.154,-0.3) node{$b$};
\draw (0,2.2) node{$c$};
\draw (-1,1) node{$\xi:=$};
\end{scope}
\end{tikzpicture}
\caption{}
\end{subfigure}
\caption{} \label{}
\end{figure}

The restriction of the $i$-th equivariant Chern class $c_i^{T}$ to a vertex $p$ is the $i$-th elementary symmetric polynomials of the weights at $p$ and $c_i^{T}$ is determined by its restriction to all the vertices. Noting this fact, one can easily check the following.
\[
c^{T}(M)=(1+\xi)(1+\xi-a)(1+\xi-\b)(1+\xi-\c). 
\]

\noindent
{\bf Claim}. 
$H^*_{T}(M)=H^*(BT)[\xi]/\big(\xi(\xi-\a)(\xi-\b)(\xi-\c)\big).$
\begin{proof}
We shall show that any element $\rho\in H^*_T(M)$ can be expressed as a polynomial in $\xi$ over $H^*(BT)$. We may assume that $\rho$ vanishes at the center vertex by subtracting the value of $\rho$ at the center vertex from $\rho$. Then the value of $\rho$ at the bottom left vertex is divisible by $\a$, so it is of the form $\alpha \a$ with some $\alpha\in H^*(BT)$. Subtracting $\alpha\xi$ from $\rho$ (note that $\alpha\xi$ vanishes at the center vertex), we may assume that $\rho$ vanishes at both the center vertex and the bottom left vertex. Then the value of $\rho$ at the bottom right vertex is divisible by $\b(\b-\a)$, so it is of the form $\beta\b(\b-\a)$ with some $\beta\in H^*(BT)$. Subtracting $\beta\xi(\xi-\a)$ from $\rho$, we may assume that $\rho$ vanishes at vertices except the top vertex (note that $\xi(\xi-\a)$ vanishes at the center vertex and left vertex). Now the value of $\rho$ at the top vertex is divisible by $\c$, $\c-\a$, and $\c-\b$, so it is of the form $\gamma \c(\c-\a)(\c-\b)$ with some $\gamma\in H^*(BT)$. Then $\rho$ agrees with $\gamma\xi(\xi-\a)(\xi-\b)$ and this proves the desired assertion. 
The element $\xi(\xi-\a)(\xi-\b)(\xi-\c)$ vanishes because its restrictions to all the vertices vanish. Therefore, we obtain a surjective homomorphism 
\[
\Phi\colon H^*(BT)[\xi]/\big(\xi(\xi-\a)(\xi-\b)(\xi-\c)\big)\to H^*_{T}(M).
\]
One can check that their Hilbert series are the same. This implies that $\Phi$ is indeed an isomorphism. 
\end{proof}

Then
\[
H^*(M)=\Z[x]/(x^4),
\]
\[
c_1(M)=4x,\quad c_2(M)=6x^2,\quad c_3(M)=4x^3. 
\]
Here, $x$ is the restriction of $\xi\in H^2_T(M)$ to $H^2(M)$.
Moreover, since the first Pontryagin class $p_1$ satisfies $p_1=c_1^2-2c_2$, the first Pontryagin class $p_1(M)$ of $M$ is $p_1(M)=4x^2$.

\subsection{Realization} \label{sec:p1}

For $u \in H^2(BT)$, we define
$$g^u:=\chi^u(g)$$
for all $g \in T$.

Let $T$ act on the complex projective space $(\C P^3,J_{\textrm{std}})$ by
$$ g \cdot [z_0:z_1:z_2:z_3]=[z_0: g^a z_1:g^b z_2:g^c z_3]$$
for all $g \in T$, 
for some non-zero $a,b,c$ that are mutually distinct and generate $H^2(BT)$.
The action has $4$ fixed points
$$[1:0:0:0],[0:1:0:0],[0:0:1:0],[0:0:0:1]$$
that have weights
$$\{a,b,c\}, \{-a,b-a,c-a\}, \{-b,a-b,c-b\}, \{-c,a-c,b-c\},$$
respectively. This is a GKM manifold with GKM graph Figure~\subref{graph1}.

\section{Type (P2)}

In this section, let $M$ denote a GKM manifold of Type (P2), that is, Figure~\subref{graph4} is the GKM graph of $M$.

\subsection{Equivariant cohomology and Chern classes}


From the GKM graph

\begin{figure}[H]
\begin{subfigure}[b][4.3cm][s]{.4\textwidth}
\centering
\begin{tikzpicture}[state/.style ={circle, draw}]
\begin{scope}[xscale=1.3, yscale=1.3]
\draw[->-=.6] (-1.154,0) to (0,0.666);
\draw[->-=.6] (1.154,0) to (0,0.666);
\draw[->-=.6] (0,2) to (0,0.666);
\draw[->-=.5] (-1.154,0) to (1.154,0);
\draw[->-=.5] (1.154,0) to (0,2);
\draw[->-=.5] (0,2) to (-1.154,0);
\draw (0,-0.3) node{$b-a$};
\draw (-0.3,0.3) node{$a$};
\draw (0.3,0.3) node{$b$};
\draw (0,0.9) node{$-a-b$};
\draw (1.1,1.3) node{$-a-2b$};
\draw (-1.1,1.3) node{$2a+b$};
\draw (0,2.3) node{};
\end{scope}
\end{tikzpicture}
\caption{}
\end{subfigure}
\begin{subfigure}[b][4.3cm][s]{.4\textwidth}
\centering
\begin{tikzpicture}[state/.style ={circle, draw}]
\begin{scope}[xscale=1.3, yscale=1.3]
\draw[->-=.6] (-1.154,0) to (0,0.666);
\draw[->-=.6] (1.154,0) to (0,0.666);
\draw[->-=.6] (0,2) to (0,0.666);
\draw[->-=.5] (-1.154,0) to (1.154,0);
\draw[->-=.5] (1.154,0) to (0,2);
\draw[->-=.5] (0,2) to (-1.154,0);
\draw (0.2,0.9) node{$0$};
\draw (-1.154,-0.3) node{$a$};
\draw (1.154,-0.3) node{$b$};
\draw (0,2.2) node{$-a-b$};
\draw (-1,1) node{$\xi:=$};
\end{scope}
\end{tikzpicture}
\caption{}
\end{subfigure}
\caption{}
\end{figure}

\noindent
we obtain
\[
\begin{split}
c_1^T(M)&=-2\xi,\quad c_2^T(M)=-\a^2-\b^2-\a\b,\\
c_3^T(M)&=\xi(\xi-\a)(\xi-\b)+\xi(\xi-\a)(\xi+\a+\b)\\
&\quad+\xi(\xi-\b)(\xi+\a+\b)+(\xi-\a)(\xi-\b)(\xi+\a+\b)\\
&=4\xi^3-2(\a^2+\b^2+\a\b)\xi \, + \, \a\b(\a+\b).
\end{split}
\]
By a similar argument to Subsection~\ref{sec:p1-coho},
\[
H^*_T(M)=H^*(BT)[\xi]/\big(\xi(\xi-\a)(\xi-\b)(\xi+\a+\b)\big).
\]
Therefore, 
\[
H^*(M)=\Z[x]/(x^4),
\]
\[
c_1(M)=-2x,\quad c_2(M)=0,\quad c_3(M)=4x^3, \quad p_1(M)=4x^2.
\]
Here, $x$ is the restriction of $\xi\in H^2_T(M)$ to $H^2(M)$.

\subsection{Realization} \label{sec:p3}

As an almost complex manifold, the $6$-sphere $S^6$ admits a $T$-action with $2$ fixed points that have weights $\{\alpha,\beta, -\alpha-\beta \}$ and $\{-\alpha,-\beta , \alpha+\beta \}$, where $\alpha,\beta$ form a basis of $H^2(BT)$, 
see Example~\ref{exa:s6}.
Its GKM graph is Figure~\ref{s6}.

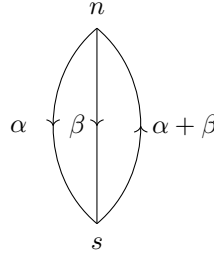
\begin{figure}[H]
\centering
\begin{tikzpicture}[state/.style ={circle, draw}]
\begin{scope}[xscale=1.3, yscale=1.3]
\draw[->-=.5] (0,2) to [bend right=50] (0,0);
\draw[->-=.5] (0,2) to (0,0);
\draw[->-=.5] (0,0) to [bend right=50]  (0,2);
\draw (-0.8,1) node{$\alpha$};
\draw (-0.2,1) node{$\beta$};
\draw (0.9,1) node{$\alpha+\beta$};
\draw (0,2.2) node{$n$};
\draw (0,-0.2) node{$s$};
\end{scope}
\end{tikzpicture}
\caption{GKM graph of $S^6$} \label{s6}
\end{figure}

We blow up a $T$-fixed point, say the north pole $n$, in $S^6$, where the blow up is in the sense of Appendix~\ref{sec:deform}. 
Figure~\ref{fig-bl} illustrates the blow up of $S^6$ at $n$ on a graph level, which corresponds to a vertex cut of $n$.

\begin{figure}[H]
\begin{subfigure}[b][4.7cm][s]{.28\textwidth}
\centering
\begin{tikzpicture}[state/.style ={circle, draw}]
\begin{scope}[xscale=1, yscale=1.3]
\draw[->-=.5] (0,1) to (-1.5,-0.5);
\draw[->-=.5] (0,1) to (0,-1);
\draw[->-=.5] (1.5,-0.5) to (0,1);
\draw (-0.7,0.5) node{$\alpha$};
\draw (0.9,0.5) node{$\alpha+\beta$};
\draw (0.2,0) node{$\beta$};
\draw (0,1.2) node{$n$};
\draw (0,1.7) node{};
\end{scope}
\end{tikzpicture}
\vspace*{4mm}
\caption{$n$} \label{}
\end{subfigure}
\begin{subfigure}[b][4.7cm][s]{.35\textwidth}
\centering
\begin{tikzpicture}[state/.style ={circle, draw}]
\begin{scope}[xscale=1.2, yscale=1.7]
\draw[->-=.5, dashed] (0,1) to (-1,0);
\draw[->-=.5] (-1,0) to (-1.5,-0.5);
\draw[->-=.5] (1.5,-0.5) to (1,0);
\draw[->-=.5, dashed] (1,0) to (0,1);
\draw[->-=.4, dashed] (0,1) to (0,-0.5);
\draw[->-=0.5] (0,-0.5) to (0,-1);
\draw[->-=.5] (-1,0) to (0,-0.5);
\draw[->-=.5] (0,-0.5) to (1,0);
\draw[->-=0.4] (1,0) to (-1,0);
\draw (0,1.2) node{$n$};
\draw (-0.7,0.5) node{$\alpha$};
\draw (-1.5,-0.2) node{$\alpha$};
\draw (1,0.5) node{$\alpha+\beta$};
\draw (1.7,-0.2) node{$\alpha+\beta$};
\draw (-0.2,0.5) node{$\beta$};
\draw (0.2,-0.9) node{$\beta$};
\draw (-0.7,-0.5) node{$\beta-\alpha$};
\draw (0.7,-0.5) node{$-\alpha-2\beta$};
\draw (0.2,0.1) node{$2\alpha+\beta$};
\end{scope}
\end{tikzpicture}
\caption{blow up of $n$} \label{}
\end{subfigure}
\begin{subfigure}[b][4cm][s]{.35\textwidth}
\centering
\begin{tikzpicture}[state/.style ={circle, draw}]
\begin{scope}[xscale=1.2, yscale=2]
\draw[->-=.5] (-1,0) to (-1.5,-0.5);
\draw[->-=.5] (1.5,-0.5) to (1,0);
\draw[->-=0.5] (0,-0.5) to (0,-1);
\draw[->-=.5] (-1,0) to (0,-0.5);
\draw[->-=.5] (0,-0.5) to (1,0);
\draw[->-=0.5] (1,0) to (-1,0);
\draw (-1.5,-0.2) node{$\alpha$};
\draw (1.7,-0.2) node{$\alpha+\beta$};
\draw (0.2,-0.9) node{$\beta$};
\draw (-0.7,-0.5) node{$\beta-\alpha$};
\draw (0.7,-0.5) node{$-\alpha-2\beta$};
\draw (0,0.2) node{$2\alpha+\beta$};
\end{scope}
\end{tikzpicture}
\vspace*{2mm}
\caption{$\textrm{Bl}_pS^6$} \label{}
\end{subfigure}
\caption{blow up of $S^6$ at $n$} \label{fig-bl}
\end{figure}
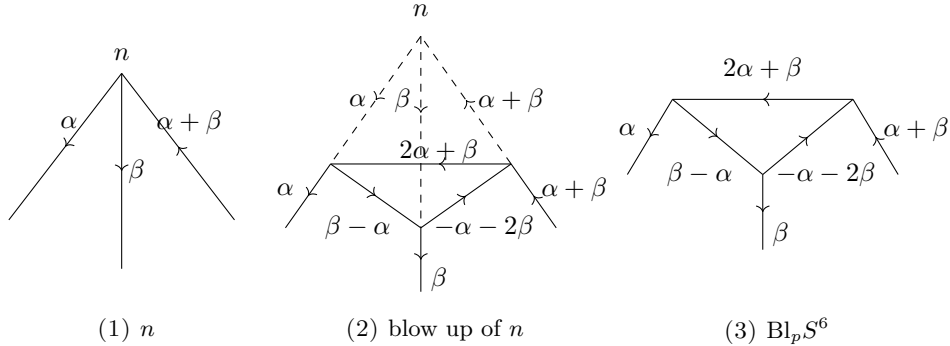

Then the GKM graph of the resulting almost complex $T$-manifold $\textrm{Bl}_pS^6$ is as follows, where the vertex in the center corresponds to the south pole $s$ of $S^6$. 

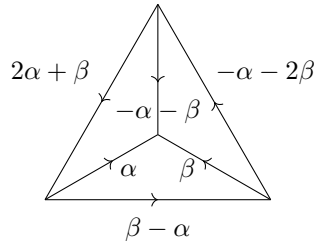
\begin{figure}[H]
\begin{tikzpicture}[state/.style ={circle, draw}]
\begin{scope}[xscale=1.3, yscale=1.3]
\draw[->-=.6] (-1.154,0) to (0,0.666);
\draw[->-=.6] (1.154,0) to (0,0.666);
\draw[->-=.6] (0,2) to (0,0.666);
\draw[->-=.5] (-1.154,0) to (1.154,0);
\draw[->-=.5] (1.154,0) to (0,2);
\draw[->-=.5] (0,2) to (-1.154,0);
\draw (0,-0.3) node{$\beta-\alpha$};
\draw (-0.3,0.3) node{$\alpha$};
\draw (0.3,0.3) node{$\beta$};
\draw (0,0.9) node{$-\alpha-\beta$};
\draw (1.1,1.3) node{$-\alpha-2\beta$};
\draw (-1.1,1.3) node{$2\alpha+\beta$};
\draw (0,2.3) node{};
\end{scope}
\end{tikzpicture}
\caption{GKM graph of $\textrm{Bl}_pS^6$}
\end{figure}

Setting $\alpha=a$ and $\beta=b$, we obtain the GKM graph of Type (P2), Figure~\subref{graph4}.

\section{Type (P3)}

In this section, let $M$ denote a GKM manifold of Type (P3), that is, Figure~\subref{graph3} is the GKM graph of $M$.

\subsection{Equivariant cohomology and Chern classes}

From the GKM graph

\begin{figure}[H]
\centering
\begin{subfigure}[b][4.3cm][s]{.4\textwidth}
\centering
\begin{tikzpicture}[state/.style ={circle, draw}]
\begin{scope}[xscale=1.3, yscale=1.3]
\draw[->-=.5] (0,0) to (2,0);
\draw[->-=.5] (0,0) to (0,2);
\draw[->-=.7] (0,0) to (2,2);
\draw[->-=.7] (2,0) to (0,2);
\draw[->-=.5] (2,0) to (2,2);
\draw[->-=.5] (0,2) to (2,2);
\draw (1,-0.3) node{$a$};
\draw (-0.2,1) node{$b$};
\draw (0.5,0.6) node{$-a-b$};
\draw (1.5,0.6) node{$a-b$};
\draw (2.2,1) node{$b$};
\draw (1,2.3) node{$a$};
\end{scope}
\end{tikzpicture}
\caption{}
\end{subfigure}
\begin{subfigure}[b][4.3cm][s]{.4\textwidth}
\centering
\begin{tikzpicture}[state/.style ={circle, draw}]
\begin{scope}[xscale=1.3, yscale=1.3]
\draw[->-=.5] (0,0) to (2,0);
\draw[->-=.5] (0,0) to (0,2);
\draw[->-=.7] (0,0) to (2,2);
\draw[->-=.7] (2,0) to (0,2);
\draw[->-=.5] (2,0) to (2,2);
\draw[->-=.5] (0,2) to (2,2);
\draw (-0.6,1) node{$\xi:=$};
\draw (0,-0.3) node{$0$};
\draw (2,-0.3) node{$a$};
\draw (0,2.3) node{$b$};
\draw (2,2.3) node{$a+b$};
\end{scope}
\end{tikzpicture}
\caption{}
\end{subfigure}
\caption{}
\end{figure}

\noindent
we obtain 
\[
\begin{split}
c_1^T(M)&=0,\quad
c_2^T(M)=-2\xi^2+2(a+b)\xi-a^2-b^2-ab,\\
c_3^T(M)&=\xi(\xi-\a)(\xi-\b)+\xi(\xi-\a)(\xi-\a-\b)\\
&\quad+\xi(\xi-\b)(\xi-\a-\b)+(\xi-\a)(\xi-\b)(\xi-\a-\b)\\
&=4\xi^3-6(\a+\b)\xi^2+2(\a^2+\b^2+3\a\b)\xi-\a\b(\a+\b).
\end{split}
\]

\[
H^*_T(M)=H^*(BT)[\xi]/\big(\xi(\xi-\a)(\xi-\b)(\xi-\a-\b)\big).
\]
Therefore,
\[
H^*(M)=\Z[x]/(x^4),
\]
\[
c_1(M)=0,\quad c_2(M)=-2x^2,\quad c_3(M)=4x^3, \quad p_1(M)=4x^2.
\]
Here, $x$ is the restriction of $\xi\in H^2_T(M)$ to $H^2(M)$.

\subsection{Realization} \label{sec:p2}

Let $G$ be a compact connected Lie group and $H$ a closed subgroup of $G$.  Let $\mathfrak{g}$ and $\mathfrak{h}$ be the Lie algebra of $G$ and $H$ respectively. As is well-known, the bundle isomorphism $G\times\mathfrak{g}\to TG$ sending $(g,v)$ to $g_*v$ induces a bundle isomorphism $\Psi\colon G\times_H(\mathfrak{g}/\mathfrak{h})\to T(G/H)$, where the $H$-action on $\mathfrak{g}/\mathfrak{h}$ is the adjoint action.

Let $J_0$ be an $H$-equivariant complex structure on $\mathfrak{g}/\mathfrak{h}$.  Then $J_0$ induces an almost complex structure $J$ on $G\times_H(\mathfrak{g}/\mathfrak{h})\cong T(G/H)$.  Indeed, we define $J$ by 
\[
J([g,v]):=[g,J_0(v)]\quad (g\in G,\ v\in \mathfrak{g}/\mathfrak{h}).
\]
This is well-defined because
\[
J([gh^{-1},hv])=[gh^{-1},J_0(hv)]=[gh^{-1},hJ_0( v)]=[g, J_0(v)]\quad (\forall h\in H).
\]

We take $G=\Sp(2)$, where 
\[
\Sp(2)=\{A\in M_2(\H)\mid {}^t\bar{A}A=I\}.
\]
Here, $M_2(\H)$ denotes the set of all square matrices of order $2$ with quaternions as entries and 
$I$ denotes the identity matrix in $M_2(\H)$. 
The Lie algebra $\mathfrak{sp}(2)$ of $\Sp(2)$ is given by
\[
\mathfrak{sp}(2)=\{ A\in M_2(\H)\mid {}^t\bar{A}=-A\}.
\]
We take $H$ to be a subgroup of $\Sp(2)$ defined by 
\[
H=\left\{  \begin{pmatrix} c&0\\
0&d\end{pmatrix}\mid c\in \C,\ d\in \H\text{ with $|c|=|d|=1$}\right\}, 
\]
which is isomorphic to $\U(1)\times \Sp(1)$.  Then $\mathfrak{g}/\mathfrak{h}$ can be identified with 
\begin{equation} \label{eq:g/h}
\mathfrak{g}/\mathfrak{h}=\left\{ \begin{pmatrix} zj& w\\
-\overline{w}&0\end{pmatrix} \mid z\in \C,\ w\in \H\right\}
\end{equation}
as $H$-module.  
Since the $H$-action on $\mathfrak{g}/\mathfrak{h}$ is the adjoint action, one can see that $\mathfrak{g}/\mathfrak{h}$ splits into two irreducible $H$-modules 
\begin{equation*} 
U_1=\left\{\begin{pmatrix} zj& 0\\
0&0\end{pmatrix} \mid z\in \C\right\},\qquad U_2=\left\{\begin{pmatrix} 0& w\\
-\overline{w}&0\end{pmatrix} \mid w\in \H\right\}.
\end{equation*}
Therefore, there are essentially two choices of $H$-equivariant complex structures $J_+$ and $J_-$ on $\mathfrak{g}/\mathfrak{h}$ defined by
\begin{enumerate}
\item $J_+\begin{pmatrix} zj& w\\
-\overline{w}&0\end{pmatrix}=\begin{pmatrix} izj& iw\\
-\overline{iw}&0\end{pmatrix}$,
\item $J_-\begin{pmatrix} zj& w\\
-\overline{w}&0\end{pmatrix}=\begin{pmatrix} -izj& iw\\
-\overline{iw}&0\end{pmatrix}$.
\end{enumerate}
A direct computation shows that both $J_+$ and $J_-$ are $H$-equivariant.   

We take $T$ to be a maximal torus of $\Sp(2)$ consisting of diagonal matrices.  
The adjoint action of $\begin{pmatrix}g_1&0\\
0&g_2\end{pmatrix}\in T$ on the element in \eqref{eq:g/h} is given by
$
\begin{pmatrix}g_1^2zj&g_1w\bar{g}_2\\
-g_2\bar{w}\bar{g}_1&0\end{pmatrix}
$ and writing $w=w_1+w_2j$ with $w_1,w_2\in\C$, we have  
$g_1w\bar{g}_2=g_1\bar{g}_2w_1+g_1g_2w_2j$. Therefore, the weights of the $T$-module $\mathfrak{g}/\mathfrak{h}$  are $2a, a-b,a+b$ for $J_+$ and $-2a, a-b,a+b$ for $J_-$, where $a,b$ are a certain basis of $H^2(BT)$. This action is not effective; $-I$ acts trivially.

The $T$-manifold $\textrm{Sp}(2)/(\U(1) \times \textrm{Sp}(1))$ with the almost complex structure induced from $J_-$ is a GKM manifold with GKM graph Figure~\ref{GKMSp(2)_v2}. 

\begin{figure}[H]
\centering
\begin{tikzpicture}[state/.style ={circle, draw}]
\begin{scope}[xscale=1.2, yscale=1.2]
\draw[->-=.5] (0,0) to (2,0);
\draw[->-=.5] (0,0) to (0,2);
\draw[->-=.7] (0,0) to (2,2);
\draw[->-=.7] (2,0) to (0,2);
\draw[->-=.5] (2,0) to (2,2);
\draw[->-=.5] (0,2) to (2,2);
\draw (1,-0.2) node{$a+b$};
\draw (-0.5,1) node{$a-b$};
\draw (0.5,0.3) node{$-2a$};
\draw (1.5,0.3) node{$2b$};
\draw (2.5,1) node{$a-b$};
\draw (1,2.2) node{$a+b$};
\end{scope}
\end{tikzpicture}
\caption{GKM graph of $\textrm{Sp}(2)/(\U(1) \times \Sp(1))$ with almost complex structure induced from $J_-$}
\label{GKMSp(2)_v2}
\end{figure}
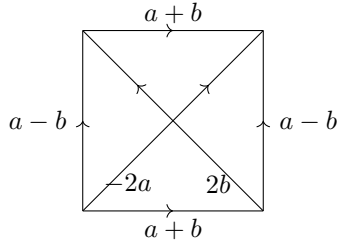

Replacing $a+b$ and $a-b$ with $a$ and $b$, respectively, this gives a realization of manifold of Type (P3).

\begin{remark} \label{rem:p3j+}
The almost complex structure on $\textrm{Sp}(2)/(\U(1) \times \textrm{Sp}(1))$ induced from $J_+$ above has the same Chern classes as the standard complex structure on $\C P^3$.
\end{remark}

\section{Type (Q1)}

In this section, let $M$ denote a GKM manifold of Type (Q1), that is, Figure~\subref{graph2} is the GKM graph of $M$.

\subsection{Equivariant cohomology and Chern classes}

From the GKM graph
\begin{figure}[H]
\begin{subfigure}[b][3.7cm][s]{.32\textwidth}
\centering
\begin{tikzpicture}[state/.style ={circle, draw}]
\draw[->-=.5] (0,0) to (2,0);
\draw[->-=.7] (0,0) to (0,2);
\draw[->-=.7] (0,0) to (2,2);
\draw[->-=.7] (2,0) to (0,2);
\draw[->-=.7] (2,0) to (2,2);
\draw[->-=.5] (0,2) to (2,2);
\draw (1,-0.3) node{$a-b$};
\draw (-0.1,1) node{$a+b$};
\draw (0.5,0.5) node{$a$};
\draw (1.5,0.5) node{$b$};
\draw (2.1,1) node{$a+b$};
\draw (1,2.3) node{$a-b$};
\end{tikzpicture}
\caption{} \label{figur-q1}
\end{subfigure}
\begin{subfigure}[b][3.7cm][s]{.32\textwidth}
\centering
\begin{tikzpicture}[state/.style ={circle, draw}]
\draw[->-=.5] (0,0) to (2,0);
\draw[->-=.5] (0,0) to (0,2);
\draw[->-=.7] (0,0) to (2,2);
\draw[->-=.7] (2,0) to (0,2);
\draw[->-=.5] (2,0) to (2,2);
\draw[->-=.5] (0,2) to (2,2);
\draw (-0.5,1) node{$\xi:=$};
\draw (0,-0.3) node{$a$};
\draw (2,-0.3) node{$b$};
\draw (0,2.3) node{$-b$};
\draw (2,2.3) node{$-a$};
\end{tikzpicture}
\caption{}
\end{subfigure}
\begin{subfigure}[b][3.7cm][s]{.32\textwidth}
\centering
\begin{tikzpicture}[state/.style ={circle, draw}]
\draw[->-=.5] (0,0) to (2,0);
\draw[->-=.5] (0,0) to (0,2);
\draw[->-=.7] (0,0) to (2,2);
\draw[->-=.7] (2,0) to (0,2);
\draw[->-=.5] (2,0) to (2,2);
\draw[->-=.5] (0,2) to (2,2);
\draw (-0.5,1) node{$\eta:=$};
\draw (0,-0.3) node{$a(a+b)$};
\draw (2,-0.3) node{$b(a+b)$};
\draw (0,2.3) node{$0$};
\draw (2,2.3) node{$0$};
\end{tikzpicture}
\caption{}
\end{subfigure}
\caption{}
\end{figure}
\noindent we obtain
\[
\begin{split}
c_1^T(M)&=3\xi,\quad c_2^T(M)=4\xi^2-\a^2-\b^2,\\
c_3^T(M)&=\xi(\xi+\a)(\xi-\a)+\xi(\xi+\b)(\xi-\b)\\
&=2\xi^3-(\a^2+\b^2)\xi.
\end{split}
\]
\begin{center}
$H^*_T(M)=H^*(BT)[\xi,\eta]/\big(2\eta-(\xi+\a)(\xi+\b) , \eta(\eta-(\a+\b) \xi)  \big).$
\end{center}
Therefore, 
\[
H^*(M)=\Z[x,y]/(2y-x^2, \ y^2),
\]
\[ 
c_1(M)=3x,\quad c_2(M)=4x^2,\quad c_3(M)=2x^3, \quad p_1(M)=x^2.
\]
Here, $x$ and $y$ are the restrictions of $\xi\in H^2_T(M)$ and $\eta\in H^4_T(M)$ to $H^2(M)$ and $H^4(M)$, respectively.

\subsection{Realization} \label{sec:q1}

Let $T$ act on the complex quadric $Q_3$ 
$$(Q_3,J_{\textrm{std}})=\{[z_0:z_1:z_2:z_3:z_4] \in \C P^4 \mid z_0z_1+z_2z_3+z_4^2=0\}$$ 
with the standard complex structure $J_{\textrm{std}}$ by
$$g \cdot [z_0:z_1:z_2:z_3:z_4]=[g^a z_0: g^{-a} z_1:g^b z_2:g^{-b} z_3:z_4]$$
for all $g \in T$, for some $a,b \in H^2(BT)$ that form a basis. 
The action has $4$ fixed points
$$[1:0:0:0:0],[0:1:0:0:0], [0:0:1:0:0], [0:0:0:1:0]$$
that have weights
$$\{-a,b-a,-b-a\}, \{a,a+b,a-b\}, \{-b,a-b,-a-b\}, \{b,a+b,b-a\},$$
respectively. This is a GKM manifold with GKM graph Figure~\subref{graph2}.

\section{Type (Q2)}

In this section, let $M$ denote a GKM manifold of Type (Q2), that is, Figure~\subref{graph5} is the GKM graph of $M$.

\subsection{Equivariant cohomology and Chern classes} \label{sec:q2coho}

From the GKM graph
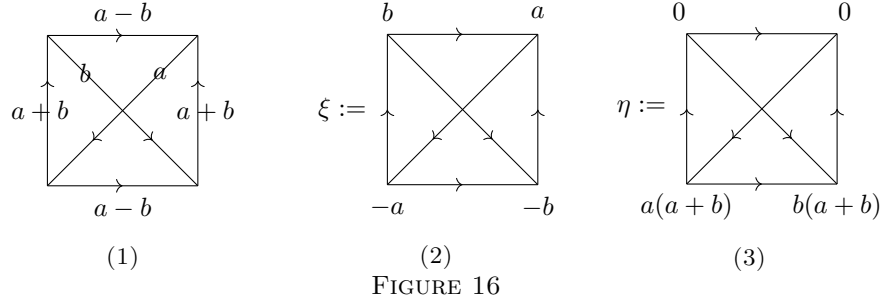
\begin{figure}[H]
\begin{subfigure}[b][3.7cm][s]{.32\textwidth}
\centering
\begin{tikzpicture}[state/.style ={circle, draw}]
\draw[->-=.5] (0,0) to (2,0);
\draw[->-=.7] (0,0) to (0,2);
\draw[->-=.7] (2,2) to (0,0);
\draw[->-=.7] (0,2) to (2,0);
\draw[->-=.7] (2,0) to (2,2);
\draw[->-=.5] (0,2) to (2,2);
\draw (1,-0.3) node{$a-b$};
\draw (-0.1,1) node{$a+b$};
\draw (0.5,1.5) node{$b$};
\draw (1.5,1.5) node{$a$};
\draw (2.1,1) node{$a+b$};
\draw (1,2.3) node{$a-b$};
\end{tikzpicture}
\caption{} \label{figur-q2}
\end{subfigure}
\begin{subfigure}[b][3.7cm][s]{.32\textwidth}
\centering
\begin{tikzpicture}[state/.style ={circle, draw}]
\draw[->-=.5] (0,0) to (2,0);
\draw[->-=.5] (0,0) to (0,2);
\draw[->-=.7] (2,2) to (0,0);
\draw[->-=.7] (0,2) to (2,0);
\draw[->-=.5] (2,0) to (2,2);
\draw[->-=.5] (0,2) to (2,2);
\draw (-0.6,1) node{$\xi:=$};
\draw (0,-0.3) node{$-a$};
\draw (2,-0.3) node{$-b$};
\draw (0,2.3) node{$b$};
\draw (2,2.3) node{$a$};
\end{tikzpicture}
\caption{}
\end{subfigure}
\begin{subfigure}[b][3.7cm][s]{.32\textwidth}
\centering
\begin{tikzpicture}[state/.style ={circle, draw}]
\draw[->-=.5] (0,0) to (2,0);
\draw[->-=.5] (0,0) to (0,2);
\draw[->-=.7] (2,2) to (0,0);
\draw[->-=.7] (0,2) to (2,0);
\draw[->-=.5] (2,0) to (2,2);
\draw[->-=.5] (0,2) to (2,2);
\draw (-0.6,1) node{$\eta:=$};
\draw (0,-0.3) node{$a(a+b)$};
\draw (2,-0.3) node{$b(a+b)$};
\draw (-0.1,2.3) node{$0$};
\draw (2.1,2.3) node{$0$};
\end{tikzpicture}
\caption{}
\end{subfigure}
\caption{}
\end{figure}
we obtain
\[
\begin{split}
c_1^T(M)&=-\xi,\quad c_2^T(M)=-\a^2-\b^2,\\
c_3^T(M)&=\xi(\xi+\a)(\xi-\a)+\xi(\xi+\b)(\xi-\b)\\
&=2\xi^3-(\a^2+\b^2)\xi.
\end{split}
\]
\[
H^*_T(M)=H^*(BT)[\xi,\eta]/\big(2\eta-(\xi-\a)(\xi-\b), \ \eta(\eta+(\a+\b)\xi) \big).
\]
Therefore, 
\[
H^*(M)=\Z[x,y]/(2y-x^2,\ y^2),
\]
\[
c_1(M)=-x,\quad c_2(M)=0,\quad c_3(M)=2x^3, \quad p_1(M)=x^2. 
\]
Here, $x$ and $y$ are the restrictions of $\xi\in H^2_T(M)$ and $\eta\in H^4_T(M)$ to $H^2(M)$ and $H^4(M)$, respectively.

\subsection{Realization} \label{sec:q2}

We blow up $S^6$ along the $T$-invariant $S^2$ corresponding to the edge with label $\alpha+\beta$ in Figure~\ref{s6}, where the blow up is in the sense of Appendix~\ref{sec:deform}. 
Figure~\ref{fig-bl2} illustrates the blow up of $S^6$ along the edge on a graph level, which corresponds to a cut along the edge.

In the decomposition of the normal bundle of the $S^2$ into the sum of two complex line bundles $L_1$ and $L_2$, if $L_1$ has weight $\beta$ at $n$, then the congruence relation implies that $L_1$ has weight $-\alpha$ at $s$, and if $L_2$ has weight $\alpha$ at $n$, then $L_2$ has weight $-\beta$ at $s$. Therefore, in Figure~\ref{fig-bl2a}, the top right (left) edge with label $\alpha$ ($\beta$) and the bottom right (left) edge with label $-\beta$ ($-\alpha$) represent fibers in the same line bundle.

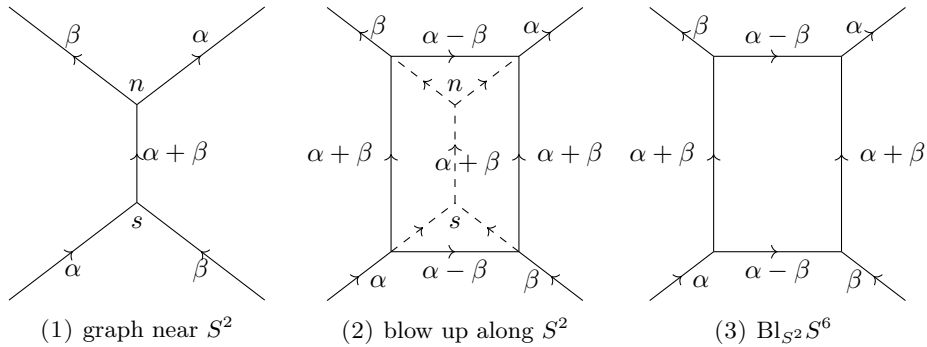
\begin{figure}[H]
\begin{subfigure}[b][4.7cm][s]{.32\textwidth}
\centering
\begin{tikzpicture}[state/.style ={circle, draw}]
\begin{scope}[xscale=1.7, yscale=1.3]
\draw[->-=.5] (0,0) to (0,1);
\draw[->-=.5] (-1,-1) to (0,0);
\draw[->-=.5] (1,-1) to (0,0);
\draw[->-=.5] (0,1) to (-1,2);
\draw[->-=.5] (0,1) to (1,2);
\draw (-0.5,-0.7) node{$\alpha$};
\draw (-0.5,1.7) node{$\beta$};
\draw (0.3,0.5) node{$\alpha+\beta$};
\draw (0.5,-0.7) node{$\beta$};
\draw (0.5,1.7) node{$\alpha$};
\draw (0,1.2) node{$n$};
\draw (0,-0.2) node{$s$};
\end{scope}
\end{tikzpicture}
\caption{graph near $S^2$} \label{fig-bl2a}
\end{subfigure}
\begin{subfigure}[b][4.7cm][s]{.33\textwidth}
\centering
\begin{tikzpicture}[state/.style ={circle, draw}]
\begin{scope}[xscale=1.7, yscale=1.3]
\draw[->-=.6,dashed] (0,0) to (0,1);
\draw[->-=.5] (-1,-1) to (-0.5,-0.5);
\draw[->-=.5,dashed] (-0.5,-0.5) to (0,0);
\draw[->-=.5] (1,-1) to (0.5,-0.5);
\draw[->-=.5,dashed] (0.5,-0.5) to (0,0);
\draw[->-=.5,dashed] (0,1) to (-0.5,1.5);
\draw[->-=.5] (-0.5,1.5) to (-1,2);
\draw[->-=.5,dashed] (0,1) to (0.5,1.5);
\draw[->-=.5] (0.5,1.5) to (1,2);
\draw[->-=.5] (-0.5,-0.5) to (0.5,-0.5);
\draw[->-=.5] (0.5,-0.5) to (0.5,1.5);
\draw[->-=.5] (-0.5,-0.5) to (-0.5,1.5);
\draw[->-=.5] (-0.5,1.5) to (0.5,1.5);
\draw (-0.6,-0.8) node{$\alpha$};
\draw (-0.6,1.8) node{$\beta$};
\draw (0.1,0.4) node{$\alpha+\beta$};
\draw (0.6,-0.8) node{$\beta$};
\draw (0.6,1.8) node{$\alpha$};
\draw (0,1.2) node{$n$};
\draw (0,-0.2) node{$s$};
\draw (0,-0.7) node{$\alpha-\beta$};
\draw (0,1.7) node{$\alpha-\beta$};
\draw (-0.9,0.5) node{$\alpha+\beta$};
\draw (0.9,0.5) node{$\alpha+\beta$};
\end{scope}
\end{tikzpicture}
\caption{blow up along $S^2$} \label{}
\end{subfigure}
\begin{subfigure}[b][4.7cm][s]{.32\textwidth}
\centering
\begin{tikzpicture}[state/.style ={circle, draw}]
\begin{scope}[xscale=1.7, yscale=1.3]
\draw[->-=.5] (-1,-1) to (-0.5,-0.5);
\draw[->-=.5] (1,-1) to (0.5,-0.5);
\draw[->-=.5] (-0.5,1.5) to (-1,2);
\draw[->-=.5] (0.5,1.5) to (1,2);
\draw[->-=.5] (-0.5,-0.5) to (0.5,-0.5);
\draw[->-=.5] (0.5,-0.5) to (0.5,1.5);
\draw[->-=.5] (-0.5,-0.5) to (-0.5,1.5);
\draw[->-=.5] (-0.5,1.5) to (0.5,1.5);
\draw (-0.6,-0.8) node{$\alpha$};
\draw (-0.6,1.8) node{$\beta$};
\draw (0.6,-0.8) node{$\beta$};
\draw (0.6,1.8) node{$\alpha$};
\draw (0,-0.7) node{$\alpha-\beta$};
\draw (0,1.7) node{$\alpha-\beta$};
\draw (-0.9,0.5) node{$\alpha+\beta$};
\draw (0.9,0.5) node{$\alpha+\beta$};
\end{scope}
\end{tikzpicture}
\caption{$\textrm{Bl}_{S^2}S^6$} \label{}
\end{subfigure}
\caption{blow up $\textrm{Bl}_{S^2}S^6$ of $S^6$ along $S^2$} \label{fig-bl2}
\end{figure}

Then the GKM graph of the resulting almost complex $T$-manifold $\textrm{Bl}_{S^2}S^6$ is Figure~\ref{fig-q2}. 
\begin{figure}[H]
\begin{tikzpicture}[state/.style ={circle, draw}]
\begin{scope}[xscale=1.3, yscale=1.3]
\draw[->-=.5] (0,0) to (2,0);
\draw[->-=.5] (0,0) to (0,2);
\draw[->-=.7] (2,2) to (0,0);
\draw[->-=.7] (0,2) to (2,0);
\draw[->-=.5] (2,0) to (2,2);
\draw[->-=.5] (0,2) to (2,2);
\draw (1,-0.2) node{$\alpha-\beta$};
\draw (-0.6,1) node{$\alpha+\beta$};
\draw (1.5,1.3) node{$\alpha$};
\draw (0.5,1.3) node{$\beta$};
\draw (2.6,1) node{$\alpha+\beta$};
\draw (1,2.2) node{$\alpha-\beta$};
\draw (0,-0.2) node{$p_1$};
\draw (2,-0.2) node{$p_4$};
\draw (0,2.2) node{$p_2$};
\draw (2,2.2) node{$p_3$};
\end{scope}
\end{tikzpicture}
\caption{} \label{fig-q2}
\end{figure}
Setting $\alpha=a$ and $\beta=b$, we obtain the GKM graph of Type (Q2), Figure~\subref{graph5}. 

\begin{remark} \label{rem:mas}
In \cite{masu84}, a closed smooth $6$-dimensional manifold $X$ is called $k$-twisted $\C P^3$ for $k\ge 1$ if $H^*(X)\cong H^*(\C P^3)$ as groups and $x^3$ for a generator $x\in H^2(X)$ is $k$-times a generator of $H^6(X)$.  It is shown in \cite[Corollary 9.2]{masu84} that if $X$ admits an almost complex effective $T$-action 
and is simply connected, then $X$ is diffeomorphic to $\C P^3$ or $Q_3$, in particular $k=1$ or $2$.  This is correct.  However, Remark (3) there saying that the total Chern class of such $X$ is of the same form as the standard $\C P^3$ or $Q_3$ is incorrect. Indeed, Types (P2,P3,Q2) (i.e. the case where Todd genus is zero) are overlooked. 
\end{remark}

\section{Type (S)} \label{sect:s}

In this section, let $M$ denote a GKM manifold of Type (S), that is, Figure~\subref{graph6} is the GKM graph of $M$.

\subsection{Equivariant cohomology and Chern classes}

Let $\{\a, \b\}$ be a basis of $H^2(BT)$. 
Let $\{\c,\d\}$ be another basis of $H^2(BT)$ such that $a+b=c+d$ and $\{a,b\}\equiv \{-c,-d\}\mod{(a+b)}$. 

\begin{figure}[H]
\begin{subfigure}[b][4cm][s]{.3\textwidth}
\begin{tikzpicture}[state/.style ={circle, draw}]
\draw[->-=.5] (2,0) to (0,0);
\draw[->-=.5] (0,0) to [bend left=20]  (0,2);
\draw[->-=.5] (0,0) to [bend right=20]  (0,2);
\draw[->-=.5] (2,2) to [bend left=20]  (2,0);
\draw[->-=.5] (2,2) to [bend right=20]  (2,0);
\draw[->-=.5] (0,2) to (2,2);
\draw (1,-0.2) node{$a+b$};
\draw (-0.4,1) node{$a$};
\draw (0.5,1) node{$b$};
\draw (1.6,1) node{$c$};
\draw (2.4,1) node{$d$};
\draw (1,2.2) node{$a+b$};
\end{tikzpicture}
\caption{$c=a-k(a+b)$, \,\,\, \, $d=b+k(a+b)$} \label{}
\end{subfigure} 
\begin{subfigure}[b][4cm][s]{.3\textwidth}
\begin{tikzpicture}[state/.style ={circle, draw}]
\draw[->-=.5] (2,0) to (0,0);
\draw[->-=.5] (0,0) to [bend left=20]  (0,2);
\draw[->-=.5] (0,0) to [bend right=20]  (0,2);
\draw[->-=.5] (2,2) to [bend left=20]  (2,0);
\draw[->-=.5] (2,2) to [bend right=20]  (2,0);
\draw[->-=.5] (0,2) to (2,2);
\draw (-0.7,1) node{$\xi:=$};
\draw (0,-0.3) node{$a+b$};
\draw (2,-0.3) node{$0$};
\draw (0,2.3) node{$a+b$};
\draw (2,2.3) node{$0$};
\end{tikzpicture}
\caption{} \label{}
\end{subfigure} \hspace{1em}
\begin{subfigure}[b][4cm][s]{.3\textwidth}
\begin{tikzpicture}[state/.style ={circle, draw}]
\draw[->-=.5] (2,0) to (0,0);
\draw[->-=.5] (0,0) to [bend left=20]  (0,2);
\draw[->-=.5] (0,0) to [bend right=20]  (0,2);
\draw[->-=.5] (2,2) to [bend left=20]  (2,0);
\draw[->-=.5] (2,2) to [bend right=20]  (2,0);
\draw[->-=.5] (0,2) to (2,2);
\draw (-0.7,1) node{$\eta:=$};
\draw (0,-0.3) node{$0$};
\draw (2,-0.3) node{$0$};
\draw (0,2.3) node{$ab$};
\draw (2,2.3) node{$cd$};
\end{tikzpicture}
\caption{} \label{}
\end{subfigure} 
\caption{}\label{}
\end{figure}

Since $\{\a,\b\}\equiv \{-\c,-\d\}\mod{(\a+\b)}$, we have $\a\b\equiv \c\d\mod{(\a+\b)}$. Therefore, there exists $\de\in H^2(BT)$ such that
\[
\a\b-\c\d=\de(\a+\b).
\] 
We obtain 
\[
\begin{split}
c_1^T(M)&=0,\quad c_2^T(M)=\de\xi+\c\d-(\a+\b)^2,\\
c_3^T(M)&=4\xi\eta-2(\a+\b)\eta-(\a\b+\c\d)\xi+(\a+\b)\c\d.
\end{split}
\]
\[
H^*_T(M)=H^*(BT)[\xi,\eta]/(\xi^2-(\a+\b)\xi,\ \eta^2-\c\d\eta-\de\xi\eta).
\]
Therefore, 
\[
H^*(M)=\Z[x,y]/(x^2,\ y^2),
\]
\[
c_1(M)=c_2(M)=0,\quad c_3(M)=4xy, \quad p_1(M)=0.
\]
Here, $x$ and $y$ are the restrictions of $\xi\in H^2_T(M)$ and $\eta\in H^4_T(M)$ to $H^2(M)$ and $H^4(M)$, respectively.

\subsection{Realization} \label{sec:s}

We discuss the existence of $M$ with GKM graph~\subref{graph6}, as equivariant gluing of two $S^6$'s along orbits in isotropy $2$-spheres.

Recall that there is a natural isomorphism 
\[
\Psi\colon H^2(BT)\cong \Hom(T,S^1).
\] 
The complex one-dimensional $T$-module defined by $\chi^u \ (u \in H^2(BT))$ is denoted by $\C(\chi^u)$ and the unit disk (resp. sphere) of a $T$-module $V$ is denoted by $D(V)$ (resp. $S(V)$). 

Let $a,b$ be a basis of $H^2(BT)$ and we denote by $S^6(a,b)$ the $S^6$ with the almost complex $T$-action with weights $a,b,-a-b$. Let $c,d$ be another basis of $H^2(BT)$ such that 
\begin{equation} \label{eq:D2_condition}
\text{$a+b=c+d$\quad and\quad $\{a,b\}\equiv \{-c,-d\}\mod a+b$.}
\end{equation}
\begin{remark} \label{rema:D2_condition}
(1) An elementary observation shows that the two conditions in \eqref{eq:D2_condition} are equivalent to the following:
\[
\{c,d\}=\{a-k(a+b),\ b+k(a+b)\}\quad \text{for some $k \in \Z$}.
\] 

\noindent (2) 
We may replace the congruence relation $\{a,b\}\equiv \{-c,-d\}\mod a+b$ in \eqref{eq:D2_condition} by $\{a,b\}\equiv \{c,d\}\mod a+b$. 
\end{remark} 

We take a $T$-orbit of a point in $S^6(a,b)$ with isotropy subgroup $K=\ker\chi^{a+b}$. We denote the $T$-orbit by $\Sab$. It is a circle on which $T/K$ acts freely. The $T$-normal bundle of $\Sab$ in $S^6(a,b)$ is trivial so that the $T$-equivariant closed tubular neighborhood $N(\Sab)$ of $\Sab$ is $T$-equivariantly diffeomorphic to 
\[
N(\Sab)\cong \Sab\times D(\R\oplus\C(\chi^a)\oplus\C(\chi^b))
\]
where $\R$ denotes the trivial real one-dimensional $T$-module. 
Condition \eqref{eq:D2_condition} ensures that $N(\Scd)$ is $T$-equivariantly diffeomorphic to $N(\Sab)$. To simplify notation, we set
\[
C:=\Sab,\quad D:= D(\R\oplus\C(\chi^a)\oplus\C(\chi^b)),\quad S:=S(\R\oplus\C(\chi^a)\oplus\C(\chi^b)).
\]
Note that $D$ is a $5$-disk and $S$ is a $4$-sphere. The $K$-fixed point set in $D$ (resp. $S$) is $D(\R)=D^1$ (resp. $S(\R)=S^0$). 

We take $T$-equivariant embeddings 
\[
\varphi_+\colon C\times D\to S^6(a,b)\quad\text{and}\quad \varphi_-\colon C\times D\to S^6(c,d)
\] 
such that $\varphi_+$ is orientation preserving while $\varphi_-$ is orientation reversing. To be more precise, $\varphi_+$ (resp. $\varphi_-$) is orientation preserving (resp. reversing) on not only the entire space but also on the $K$-fixed point set $C\times D(\R)$. We make the following natural identification 
\[
C\times (D\backslash\{0\})=C\times S\times (0,1]\quad\text{and}\quad C\times (\D\backslash\{0\})=C\times S\times (0,1).
\]
We take a self-diffeomorphism 
\[
\psi\colon C\times S\times (0,1)\to C\times S\times(0,1),\quad (c,s,t)\mapsto (c,s,1-t)
\]
and glue $S^6(a,b)\backslash C$ and $S^6(c,d)\backslash C$ through the diffeomorphism $\varphi_-\circ\psi\circ\varphi_+^{-1}$ mapping $\varphi_+(C\times S\times (0,1))$ onto $(\varphi_-\circ\psi)(C\times S\times(0,1))$. Since $\varphi_-\circ\psi\circ\varphi_+^{-1}$ is $T$-equivariant and orientation preserving, the resulting manifold, denoted by $M$, is a closed oriented $T$-manifold and has an almost complex structure inherited from $S^6(a,b)$ and $S^6(c,d)$ outside 
\[
\varphi_+(C\times S\times (0,1))=(\varphi_-\circ\psi)(C\times S\times (0,1)).
\]

\begin{figure}[H]
\begin{subfigure}[b][3.3cm][s]{.32\textwidth}
\begin{tikzpicture}[state/.style ={circle, draw}]
\begin{scope}[xscale=1.2, yscale=1.2]
\draw[->-=.5] (0,2) to [bend left=20]  (0,0);
\draw[->-=.5] (0,2) to [bend right=20]  (0,0);
\draw[->-=.5] (0,0) to [bend right=60]  (0,2);
\draw[->-=.5] (2,2) to [bend right=60]  (2,0);
\draw[->-=.5] (2,0) to [bend left=20]  (2,2);
\draw[->-=.5] (2,0) to [bend right=20]  (2,2);
\draw (0.7,0.2) node{$a+b$};
\draw (-0.4,1) node{$a$};
\draw (0,1) node{$b$};
\draw (2,1) node{$c$};
\draw (2.4,1) node{$d$};
\draw (1.2,1.8) node{$c+d$};
\end{scope}
\end{tikzpicture}
\vspace*{2mm}
\caption{$S^6(a,b) \sqcup S^6(c,d)$} \label{}
\end{subfigure} 
\begin{subfigure}[b][3.3cm][s]{.33\textwidth}
\begin{tikzpicture}[state/.style ={circle, draw}]
\begin{scope}[xscale=1.2, yscale=1.2]
\draw[->-=.5] (0,2) to [bend left=20]  (0,0);
\draw[->-=.5] (0,2) to [bend right=20]  (0,0);
\draw[->-=.5] (0,0) to [bend right=20]  (0.6,0.8);
\draw[->-=.5] (0.6,1.2) to [bend right=20]  (0,2);
\draw[->-=.5] (1.4,0.8) to [bend right=20] (2,0) ;
\draw[->-=.5] (2,2) to [bend right=20] (1.4,1.2) ;
\draw[->-=.5] (2,0) to [bend left=20]  (2,2);
\draw[->-=.5] (2,0) to [bend right=20]  (2,2);
\draw (0.7,0.2) node{$a+b$};
\draw (-0.4,1) node{$a$};
\draw (0,1) node{$b$};
\draw (2,1) node{$c$};
\draw (2.4,1) node{$d$};
\draw (1.2,1.8) node{$c+d$};
\end{scope}
\end{tikzpicture}
\vspace*{2mm}
\caption{$S^6(a,b)\backslash C \sqcup S^6(c,d)\backslash C$} \label{}
\end{subfigure} 
\begin{subfigure}[b][3.3cm][s]{.33\textwidth}
\begin{tikzpicture}[state/.style ={circle, draw}]
\begin{scope}[xscale=1.2, yscale=1.2]
\draw[->-=.5] (0,2) to [bend left=20]  (0,0);
\draw[->-=.5] (0,2) to [bend right=20]  (0,0);
\draw[->-=0.5] (0,0) to (2,0);
\draw[->-=0.5] (2,2) to (0,2);
\draw[->-=.5] (2,0) to [bend left=20]  (2,2);
\draw[->-=.5] (2,0) to [bend right=20]  (2,2);
\draw (1,0.2) node{$a+b$};
\draw (-0.4,1) node{$a$};
\draw (0,1) node{$b$};
\draw (2,1) node{$c$};
\draw (2.4,1) node{$d$};
\draw (1,1.8) node{$a+b$};
\end{scope}
\end{tikzpicture}
\caption{$M$} \label{}
\end{subfigure} 
\caption{Construction of $M$ as equivariant gluing of $S^6(a,b)\backslash C$ and $S^6(c,d)\backslash C$ with $a+b=c+d$ and $\{a,b\} \equiv \{-c,-d\} \mod a+b$}\label{graph-s:const}
\end{figure}
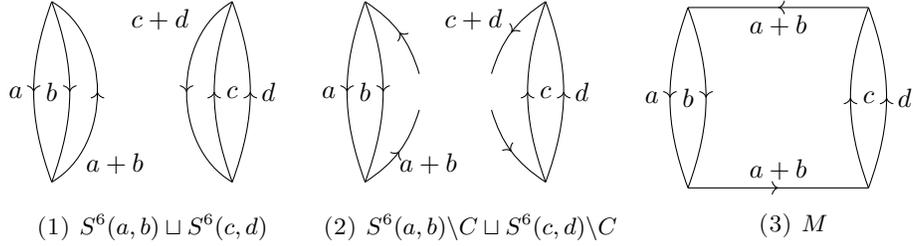

We note that $C\times S\times [0,1]$ is embedded in $M$. Indeed, the embedding $\Phi\colon C\times S\times[0,1]\to M$ is given by
\[
\Phi=\begin{cases} \varphi_+ &\text{on $C\times S\times (0,1]$}\\
\varphi_-\circ\psi&\text{on $C\times S\times [0,1)$}.\end{cases}
\] 
Note that $\Phi$ is well-defined because $\varphi_+=\varphi_-\circ\psi$ on $C\times S\times(0,1)$. One boundary $C\times S\times \{1\}$ of $C\times S\times[0,1]$ has an almost complex structure induced from $\varphi_+$ and the other one $C\times S\times\{0\}$ has it induced from $\varphi_-\circ\psi$. They are $T$-equivariant because the maps $\varphi_\pm$ and $\psi$ are $T$-equivariant. Thus, in order to obtain a $T$-equivariant almost complex structure on $M$, it suffices to show that the $T$-equivariant almost complex structure on $TM|_{(C\times S\times \{0,1\})}$ extends over $X:=C\times S\times [0,1]$, which we shall observe in the following. 

We first note that since the action of $T/K$ on $C$ is free, it suffices to show that the almost complex structure on $TX|_{(\{c\}\times S\times \{0,1\})}$ for some $c\in C$, which is $K$-equivariant, extends over $\{c\}\times S\times [0,1]$. In the following we omit $\{c\}$ to simplify notation. 

Remember that $S^K=S(\R)=S^0$. Note that $TX|_{(S(\R)\times [0,1])}$ is a $K$-vector bundle where $K$ acts trivially on the base space $S(\R)\times [0,1]$, so it splits into Whitney sum of two trivial real vector bundles according to the fiber $K$-representation: one is of rank two with trivial $K$-action and the other is of rank four with non-trivial $K$-action. Remember that $TX|_{(S(\R)\times \{0,1\})}$ is equipped with a $K$-equivariant almost complex structure and this bundle splits into Whitney sum of the trivial part for the $K$-action and the nontrivial part for the $K$-action. The complex structure on the trivial part for the $K$-action extends over $S(\R)\times[0,1]$ because those complex structures are parametrized by ${\rm GL}_2^+(\R)/{\rm GL}_1(\C)$ which is connected. As for the nontrivial part for the $K$-action, the weights of the nontrivial fiber $K$-action on the boundaries $S(\R)\times\{0,1\}$ agree, indeed they are $\{a,-a\}$ modulo $a+b$ which follows from Remark~\ref{rema:D2_condition}. Therefore, the almost complex structure on the nontrivial part for the $K$-action also extends over $S(\R)\times [0,1]$ $K$-equivariantly. 

Let 
\[Y=S\times \{0,1\}\cup S(\R)\times [0,1].
\] 
We have shown that the $K$-equivariant almost complex structure on $TX|_{(S\times\{0,1\})}$ extends to $TX|_Y$. Let $N(Y)$ be a closed $K$-invariant neighborhood of $Y$ which deforms to $Y$ $K$-equivariantly. Since $Y$ is a $K$-equivariant deformation retract of $N(Y)$, we may assume that our $K$-equivariant almost complex structure on $TX|_{Y}$ extends over $N(Y)$. 

Explicitly we take 
\[
N(Y)=S\times \left([0,1/3]\cup [2/3,1]\right)\cup D_\epsilon\times [0,1],
\]
where $D_\epsilon$ is a closed $K$-invariant small neighborhood of $S(\R)$ in $S$. Then 
\[
\begin{split}
S\times[0,1]\backslash \Int N(Y)&=\left((S\backslash \Int D_\epsilon)\times[0,1]\right) \backslash \left(S\times\left([0,1/3)\cup (2/3,1]\right)\right)\\
&=(S\backslash \Int D_\epsilon)\times [1/3,2/3].
\end{split}
\]
Since $S\backslash \Int D_\epsilon$ is $K$-equivariantly homeomorphic to $S^3\times I$, where $S^3$ is the equator of $S$ and $I$ is a closed interval, $(S\backslash \Int D_\epsilon)\times [1/3,2/3]$ is $K$-equivariantly homeomorphic to $S^3\times D^2$, where $K$ acts freely on $S^3$ and trivially on $D^2$. 
We note that $K$-equivariant almost complex structures on $TX|S^3\times \partial D^2$ are equipped and it suffices to show that it extends over $S^3\times D^2$ $K$-equivariantly. 
Since the $K$-action on $S^3\times D^2$ is free, this extension problem reduces to an extension problem on their quotients by the $K$-action. 

Now we apply the obstruction theory. The space which parametrizes almost complex structures on $\R^6$ is $\GL^+_6(\R)/\GL_3(\C)$ which is homotopy equivalent to ${\rm SO}(6)/{\rm U}(3)$, so it suffices to check 
\begin{equation} \label{eq:obstruction}
H^{q+1}(S^2\times D^2, S^2\times \partial D^2;\pi_q({\rm SO}(6)/{\rm U}(3)))=0\quad \text{for all $q\ge 0$}.
\end{equation}
An elementary computation shows that 
\[
H^{q+1}(S^2\times D^2, S^2\times \partial D^2)\cong \begin{cases} \Z \quad&\text{if $q=1,3$},\\
0\quad&\text{otherwise}.\end{cases}
\]
As noted in \cite[page 5]{KL}, Bott periodicity implies that for $q<2n-1$, 
$$\pi_q(\SO(n)/\U(n))=0 \quad \textrm{if } q \equiv 1,3,4,5 \mod 8,$$
see \cite[page 432]{G}. 
In particular, for $q=1,3$,
$\pi_q({\rm SO}(6)/{\rm U}(3))=0$ holds\footnote{This also follows from the fact that $\SO(6)/\U(3)$ is diffeomorphic to $\C P^3$.}. Therefore, \eqref{eq:obstruction} is satisfied. This completes the proof.

\section{Automorphism group} \label{sec:auto}

Let $M$ be a connected compact almost complex manifold with an almost complex structure $J$. An automorphism of $M$ means a diffeomorphism of $M$ preserving the almost complex structure $J$. It is known that the group of automorphisms of $M$, denoted by $\Aut(M)$, is a Lie group \cite{bo-ko-wa63}. 

Let $T$ be a maximal torus of $\Aut(M)$. Suppose that $M$ with the $T$-action is a GKM manifold. Let $\GM$ be its GKM graph and $\Aut(\GM)$ the automorphism group of $\GM$. 
Here an automorphism of $\GM$ is a pair $(\varphi,\phi)$ of a graph automorphism $\varphi$ of $\GM$ and an automorphism $\phi$ of $H^2(BT)$ satisfying $\alpha(\varphi(e))=\phi(\alpha(e))$ for any oriented edge $e$ of $\GM$ where $\alpha$ is the axial function on $\GM$, see \cite{JKMSZ}.
Let $\GG$ be a connected maximal compact Lie subgroup of $\Aut(M)$ and $N_\GG(T)$ the normalizer of $T$ in $\GG$. An element $\sigma\in N_\GG(T)$ naturally induces an element $\sigma_*\in \Aut(\GM)$ and $\sigma_*$ is the identity when $\sigma\in T$. Therefore, we obtain a homomorphism 
\begin{equation} \label{eq:Psi}
\Psi\colon N_\GG(T)/T\to \Aut(\GM),
\end{equation}
where $N_\GG(T)/T$ is the Weyl group of $\GG$. 

\begin{lemma}
$\Psi$ is injective.
\end{lemma}

\begin{proof}
It is clear from the definition of $\sigma_*$ that $\Psi$ is a homomorphism. Therefore, it suffices to show that $\sigma\in T$ when $\sigma_*$ is the identity. 

Suppose that $\sigma_*$ is the identity. Then $\sigma$ fixes each $T$-fixed point $p$ in $M$ and preserves all complex one-dimensional factors in the tangential $T$-module $T_pM$. This means that the differential $d\sigma_p$ of $\sigma$ at $p$ acts on each complex one-dimensional factor in $T_pM$ by a complex multiplication. Therefore, $d\sigma_p$ commutes with $dg_p$ for any $g\in T$. This implies that $\sigma$ commutes with any $g$ on a neighborhood of $p$. Indeed, this can be seen if we take a $\GG$-invariant Riemannian metric and consider the associated exponential map $T_pM\to M$ which maps the origin of $T_pM$ to $p$ and is a $\GG_p$-equivariant local diffeomorphism around the origin. Here $\GG_p$ is the isotropy subgroup of $\GG$ at $p$ and we note that $\sigma$ and $g\in T$ belong to $\GG_p$. 

Now we consider the commutator $\sigma g \sigma^{-1} g^{-1}\in \GG$. This fixes an open neighborhood of $p$ in $M$ as observed above. On the other hand, since $\GG$ is compact, the fixed point set of $\sigma g \sigma^{-1} g^{-1}$ must be a closed submanifold of $M$. Since $M$ is connected, this implies that $\sigma g \sigma^{-1} g^{-1}$ is the identity on the entire $M$, i.e. $\sigma$ commutes with any $g\in T$. Since $\sigma$ is an element of $N_\GG(T)$, this means that $\sigma\in T$. 
\end{proof}

\begin{remark}
The group $\Aut(\GM)$ naturally acts on $H^*(M)$ if $H^{\textrm{odd}}(M)=0$ through GKM theory.
We define 
\[
\Aut^*(\GM)=\{ f\in \Aut(\GM)\mid \text{the induced action of $f$ on $H^*(M)$ is trivial}\}.
\]
On the other hand, since $\GG$ is connected, the induced action of $\GG$ on $H^*(M)$ is trivial, in particular, the induced action of $\sigma\in N_\GG(T)$ on $H^*(M)$ is trivial. 
Therefore, the image of the map $\Psi$ in \eqref{eq:Psi} lies in an apparently smaller group $\Aut^*(\GM)$ than $\Aut(\GM)$. 
\end{remark}

\subsection{Type (P1)}

Let $M$ be the $\C P^3$ with the standard complex structure $J_{\textrm{std}}$.
The linear action of $\U(4)$ on $\C P^3$ preserves $J_{\textrm{std}}$. However, the center of $\U(4)$ acts on $M$ trivially and the action of $\U(4)$ on $M$ descends to an effective action of $\PU(4)$. 
Therefore, a connected maximal compact Lie subgroup $\GG(\textrm{P1})$ of $\Aut(M)$ contains $\PU(4)$. The Weyl group of $\PU(4)$ is the symmetric group $\mathfrak{S}_4$ of degree $4$.

Fix $a_1 \in H^2(BT)$. In Figure~\subref{graph1}, let $a=a_2-a_1$, $b=a_3-a_1$, $c=a_4-a_1$. For $1 \leq i < j \leq 4$, let $\sigma_{i,j}$ be a graph automorphism that permutes $p_i$ and $p_j$ and permutes $a_i$ and $a_j$. 
\begin{figure}[H]
\centering
\begin{tikzpicture}[state/.style ={circle, draw}]
\begin{scope}[xscale=1.5, yscale=1.5]
\draw[->-=.6] (0,0.666) to (-1.154,0);
\draw[->-=.6] (0,0.666) to (1.154,0);
\draw[->-=.6] (0,0.666) to (0,2);
\draw[->-=.5] (-1.154,0) to (1.154,0);
\draw[->-=.5] (1.154,0) to (0,2);
\draw[->-=.5] (0,2) to (-1.154,0);
\draw (0,-0.2) node{$a_3-a_2$};
\draw (-0.6,0.5) node{$a_2-a_1$};
\draw (0.6,0.5) node{$a_3-a_1$};
\draw (0,0.9) node{$a_4-a_1$};
\draw (0.9,1.3) node{$a_4-a_3$};
\draw (-0.9,1.3) node{$a_2-a_4$};
\draw (0,2.1) node{$p_4$};
\draw (0,0.4) node{$p_1$};
\draw (-1.154,-0.2) node{$p_2$};
\draw (1.154,-0.2) node{$p_3$};
\end{scope}
\end{tikzpicture}
\caption{} \label{}
\end{figure}
One can see that $\sigma_{i,j} \in \Aut(\GM)$ for all $1 \leq i < j \leq 4$ and $\Aut(\GM)=\mathfrak{S}_4$.
Hence the map $\Psi$ in \eqref{eq:Psi} is an isomorphism, which implies $\GG(\textrm{P1})=\PU(4)$.

\begin{remark}
It is known that the automorphism group of $\C P^{n}$ with the standard complex structure $J_{\textrm{std}}$ is ${\rm PGL}_{n+1}(\C)$ and its maximal compact connected Lie subgroup is $\PU(n+1)$. 
\end{remark}

\subsection{Type (P2)} \label{ssec:auto-p2}

Let $M$ be the manifold obtained by blowing up a $T$-fixed point (the north pole or the south pole) in $S^6$.
Indeed, $S^6$ with the standard almost complex structure has the symmetry of $G_2$ and there is a subgroup $\SU(3)$ of $G_2$ which fixes the north pole and the south pole, see Appendix~\ref{sec:octo}. Since the blow-up can be done $\SU(3)$-equivariantly, $\SU(3)$ acts on $M$ preserving the almost complex structure on $M$. This action is effective, so a connected maximal compact Lie subgroup $\GG(\textrm{P2})$ of $\Aut(M)$ contains $\SU(3)$. The Weyl group of $\SU(3)$ is the symmetric group $\mathfrak{S}_3$ of degree $3$.

Only the sum of the weights at $p_1$ is zero; therefore, any element of $\Aut(\GM)$ fixes $p_1$.
In Figure~\subref{graph4}, let $a=a_2$, $b=a_3$, $-a-b=a_4$. For $2 \leq i < j \leq 4$, let $\sigma_{i,j}$  be a graph automorphism that permutes $p_i$ and $p_j$ and permutes $a_i$ and $a_j$. 
\begin{figure}[H]
\centering
\begin{tikzpicture}[state/.style ={circle, draw}]
\begin{scope}[xscale=1.5, yscale=1.5]
\draw[->-=.6] (-1.154,0) to (0,0.666);
\draw[->-=.6] (1.154,0) to (0,0.666);
\draw[->-=.6] (0,2) to (0,0.666);
\draw[->-=.5] (-1.154,0) to (1.154,0);
\draw[->-=.5] (1.154,0) to (0,2);
\draw[->-=.5] (0,2) to (-1.154,0);
\draw (0,-0.2) node{$a_3-a_2$};
\draw (-0.6,0.5) node{$a_2$};
\draw (0.6,0.5) node{$a_3$};
\draw (0,0.9) node{$a_4$};
\draw (0.9,1.3) node{$a_4-a_3$};
\draw (-0.9,1.3) node{$a_2-a_4$};
\draw (0,2.1) node{$p_4$};
\draw (0,0.4) node{$p_1$};
\draw (-1.154,-0.2) node{$p_2$};
\draw (1.154,-0.2) node{$p_3$};
\end{scope}
\end{tikzpicture}
\caption{} \label{}
\end{figure}
One can see that $\sigma_{i,j} \in \Aut(\GM)$ for all $2 \leq i < j \leq 4$ and $\Aut(\GM)=\mathfrak{S}_3$. Hence the map $\Psi$ in \eqref{eq:Psi} is an isomorphism, which implies $\GG(\textrm{P2})=\SU(3)$.

\subsection{Type (P3)}

Let $M$ be $\Sp(2)/(\U(1) \times \Sp(1)) \cong \C P^3$ with the \emph{non-standard} almost complex structure $J_{\textrm{n-std}}$, which is induced from $J_{-}$ in Subsection~\ref{sec:p2}.
The action of the center $\pm I$ on $M$, where $I$ denotes the identity matrix in $\Sp(2)$, is trivial and the action of $\Sp(2)$ on $M$ descends to an effective action of $\Sp(2)/\{\pm I\}$. 
Therefore, a connected maximal compact Lie subgroup $\GG(\textrm{P3})$ of $\Aut(M)$ contains $\Sp(2)/\{\pm I\}$. The Weyl group of $\Sp(2)/\{\pm I\}$ is the dihedral group of order $8$.

Let $\tau$ be the counterclockwise rotation by 90 degrees of the GKM graph together with the map that sends $a$ to $b$ and sends $b$ to $-a$.
Let $\rho$ be the reflection through the middle vertical line of the GKM graph together with the map that sends $a$ to $-a$.
\begin{figure}[H]
\begin{subfigure}[b][4.8cm][s]{.4\textwidth}
\centering
\begin{tikzpicture}[state/.style ={circle, draw}]
\begin{scope}[xscale=1.2, yscale=1.2]
\draw[->-=.5] (0,0) to (2,0);
\draw[->-=.5] (0,0) to (0,2);
\draw[->-=.7] (0,0) to (2,2);
\draw[->-=.7] (2,0) to (0,2);
\draw[->-=.5] (2,0) to (2,2);
\draw[->-=.5] (0,2) to (2,2);
\draw (1,-0.2) node{$a$};
\draw (-0.2,1) node{$b$};
\draw (0.5,0.6) node{$-a-b$};
\draw (1.5,0.6) node{$a-b$};
\draw (2.2,1) node{$b$};
\draw (1,1.8) node{$a$};

\draw (0.8,2.3) node[scale=1.5]{$\circlearrowleft$};
\draw (1.2,2.3) node{$90^\circ$};


\draw (1,-0.6) node{$a \mapsto b, \ b \mapsto -a$};
\end{scope}
\end{tikzpicture}
\caption{$\tau$}
\label{}
\end{subfigure} 
\begin{subfigure}[b][4.8cm][s]{.4\textwidth}
\centering
\begin{tikzpicture}[state/.style ={circle, draw}]
\begin{scope}[xscale=1.2, yscale=1.2]
\draw[dashed] (1,-0.2) to (1,2.5);
\draw (1,2.3) node{$\longleftrightarrow$};

\draw[->-=.5] (0,0) to (2,0);
\draw[->-=.5] (0,0) to (0,2);
\draw[->-=.7] (0,0) to (2,2);
\draw[->-=.7] (2,0) to (0,2);
\draw[->-=.5] (2,0) to (2,2);
\draw[->-=.5] (0,2) to (2,2);
\draw (1,-0.2) node{$a$};
\draw (-0.2,1) node{$b$};
\draw (0.5,0.6) node{$-a-b$};
\draw (1.5,0.6) node{$a-b$};
\draw (2.2,1) node{$b$};
\draw (1,2.2) node{$a$};

\draw (1,-0.6) node{$a \mapsto -a$};
\end{scope}
\end{tikzpicture}
\caption{$\rho$}
\label{}
\end{subfigure}
\caption{Graph automorphisms of Figure~\subref{graph3}}\label{fig:graph_auto_p3}
\end{figure}

One can check that $\tau$ and $\rho$ generate $\Aut(\GM)$, which is the dihedral group of order 8.
Therefore, the map $\Psi$ in \eqref{eq:Psi} is an isomorphism, which implies $\GG(\textrm{P3})=\Sp(2)/\{\pm I\}$.

\begin{remark} \label{rem:auto-p3}
Take $M=G/K$ where $G=\Sp(2)/\{\pm I\}$. By \cite[Proposition 3.1]{wolf69} and its subsequent Remark 1, the connected maximal compact Lie subgroup of $\Aut(M,J_{\textrm{n-std}})$ agrees with the largest connected group of almost hermitian isometries of $M$, which is $\Sp(2)/\{\pm I\}$.
Since $J_{\textrm{n-std}}$ is not integrable,  the identity component of $\Aut(M,J_{\textrm{n-std}})$ agrees with the largest connected group of almost hermitian isometries of $M$ by \cite[Theorem 4.1 (3)]{wolf69}.
\end{remark}

\subsection{Type (Q1)} \label{sec:aut-q1}

Let $M$ be $Q_3$ with the standard complex structure $J_{\textrm{std}}$.
Up to a linear transformation, we can also describe $(Q_3,J_{\textrm{std}})$ as the hypersurface of $\C P^4$ defined by the equation
\[
Q_3=\{[z_0:z_1:z_2:z_3:z_4]\in \C P^4\mid z_0^2+z_1^2+z_2^2+z_3^2+z_4^2=0\}.
\]
The linear action of $\SO(5)$ on $\C P^4$ preserves $Q_3$ and the action of $\SO(5)$ on $Q_3$ is effective. Therefore, a connected maximal compact Lie subgroup $\GG(\textrm{Q1})$ of $\Aut(M)$ contains $\SO(5)$. The Weyl group of $\SO(5)$ is the dihedral group of order $8$. 

Let $\tau$ be the counterclockwise rotation by 90 degrees of the GKM graph together with the map that sends $a$ to $b$ and sends $b$ to $-a$.
Let $\rho$ be the reflection through the middle vertical line of the GKM graph together with the map that swaps $a$ and $b$.
\begin{figure}[H]
\begin{subfigure}[b][4.8cm][s]{.4\textwidth}
\centering
\begin{tikzpicture}[state/.style ={circle, draw}]
\begin{scope}[xscale=1.2, yscale=1.2]
\draw[->-=.5] (0,0) to (2,0);
\draw[->-=.7] (0,0) to (0,2);
\draw[->-=.7] (0,0) to (2,2);
\draw[->-=.7] (2,0) to (0,2);
\draw[->-=.7] (2,0) to (2,2);
\draw[->-=.5] (0,2) to (2,2);
\draw (1,0.2) node{$a-b$};
\draw (-0.1,1) node{$a+b$};
\draw (0.5,0.5) node{$a$};
\draw (1.5,0.5) node{$b$};
\draw (2.1,1) node{$a+b$};
\draw (1,1.8) node{$a-b$};

\draw (0.8,2.3) node[scale=1.5]{$\circlearrowleft$};
\draw (1.2,2.3) node{$90^\circ$};
\draw (1,-0.6) node{$a \mapsto b, \ b \mapsto -a$};
\end{scope}
\end{tikzpicture}
\caption{$\tau$}
\label{}
\end{subfigure} 
\begin{subfigure}[b][4.8cm][s]{.4\textwidth}
\centering
\begin{tikzpicture}[state/.style ={circle, draw}]
\begin{scope}[xscale=1.2, yscale=1.2]
\draw[dashed] (1,-0.2) to (1,2.5);
\draw (1,2.3) node{$\longleftrightarrow$};
\draw[->-=.5] (0,0) to (2,0);
\draw[->-=.7] (0,0) to (0,2);
\draw[->-=.7] (0,0) to (2,2);
\draw[->-=.7] (2,0) to (0,2);
\draw[->-=.7] (2,0) to (2,2);
\draw[->-=.5] (0,2) to (2,2);
\draw (1,0.2) node{$a-b$};
\draw (-0.1,1) node{$a+b$};
\draw (0.5,0.5) node{$a$};
\draw (1.5,0.5) node{$b$};
\draw (2.1,1) node{$a+b$};
\draw (1,1.8) node{$a-b$};
\draw (1,-0.6) node{$a \mapsto b, \ b \mapsto a$};
\end{scope}
\end{tikzpicture}
\caption{$\rho$}
\label{}
\end{subfigure}
\caption{Graph automorphisms of Figure~\subref{graph2}}\label{fig:graph_auto_q1}
\end{figure}
One can check that $\tau$ and $\rho$ generate $\Aut(\GM)$, which is the dihedral group of order 8.
Therefore, the map $\Psi$ in \eqref{eq:Psi} is an isomorphism, which implies $\GG(\textrm{Q1})=\SO(5)$.

\begin{remark} \label{rem:q3}
It is known (for instance, see \cite[Section 3.3, Theorem 2]{Ak}) that the automorphism group of $Q_{n}=\textrm{PO}(n+2)/\textrm{P}(\textrm{O}(n)\times \textrm{O}(2))$ with the standard complex structure $J_{\textrm{std}}$ is $\textrm{PO}(n+2;\mathbb{C})$. Here,  $\textrm{PO}(n+2):=\SO(n+2)/Z$, where $Z\subset \SO(n+2)$ is the center. It is easy to check that if $n$ is odd (resp. even), then $Z=\{I\}$ (resp. $Z=\{\pm I\}$). Therefore, in particular, the maximal  compact subgroup of $\Aut(Q_{3})$ is $\SO(5)$. The result for $K(Q1)$ also follows from this fact.
\end{remark}

\subsection{Type (Q2)} \label{ssec:auto-q2}

Let $M$ be a manifold obtained by blowing up of $S^6$ along an $S^2$ fixed by an $S^1$-subgroup of $T$, which is diffeomorphic to $Q_3$.

There is a subgroup $\SO(4)$ of $G_2$ which leaves the $S^2$ invariant, see Lemma~\ref{lemm:A1} and Remark~\ref{rem:a1}. The blow-up can be done $\SO(4)$-equivariantly and the resulting action of $\SO(4)$ on $M$ is effective, so a connected maximal compact Lie subgroup $\GG(\textrm{Q2})$ of $\Aut(M)$ contains $\SO(4)$. Note that the Weyl group of $\SO(4)$ is $S^0\times S^0$.

Let $\tau$ be the counterclockwise rotation by 90 degrees of the GKM graph together with the map that sends $a$ to $b$ and sends $b$ to $-a$.
Let $\rho$ be the reflection through the middle vertical line of the GKM graph together with the map that swaps $a$ and $b$.
One can check that $\tau$ and $\rho$ generate $\Aut(\GM)$, which is the dihedral group of order 8.

\begin{figure}[H]
\begin{subfigure}[b][4.8cm][s]{.4\textwidth}
\centering
\begin{tikzpicture}[state/.style ={circle, draw}]
\begin{scope}[xscale=1.2, yscale=1.2]
\draw[->-=.5] (0,0) to (2,0);
\draw[->-=.7] (0,0) to (0,2);
\draw[->-=.7] (2,2) to (0,0);
\draw[->-=.7] (0,2) to (2,0);
\draw[->-=.7] (2,0) to (2,2);
\draw[->-=.5] (0,2) to (2,2);
\draw (1,0.2) node{$a-b$};
\draw (-0.1,1) node{$a+b$};
\draw (0.5,1.5) node{$b$};
\draw (1.5,1.5) node{$a$};
\draw (2.1,1) node{$a+b$};
\draw (1,1.8) node{$a-b$};

\draw (0.8,2.3) node[scale=1.5]{$\circlearrowleft$};
\draw (1.2,2.3) node{$90^\circ$};

\draw (1,-0.6) node{$a \mapsto b, \ b \mapsto -a$};
\end{scope}
\end{tikzpicture}
\caption{$\tau$}
\label{}
\end{subfigure} 
\begin{subfigure}[b][4.8cm][s]{.4\textwidth}
\centering
\begin{tikzpicture}[state/.style ={circle, draw}]
\begin{scope}[xscale=1.2, yscale=1.2]
\draw[dashed] (1,-0.2) to (1,2.5);
\draw (1,2.3) node{$\longleftrightarrow$};

\draw[->-=.5] (0,0) to (2,0);
\draw[->-=.7] (0,0) to (0,2);
\draw[->-=.7] (2,2) to (0,0);
\draw[->-=.7] (0,2) to (2,0);
\draw[->-=.7] (2,0) to (2,2);
\draw[->-=.5] (0,2) to (2,2);
\draw (1,0.2) node{$a-b$};
\draw (-0.1,1) node{$a+b$};
\draw (0.5,1.5) node{$b$};
\draw (1.5,1.5) node{$a$};
\draw (2.1,1) node{$a+b$};
\draw (1,1.8) node{$a-b$};

\draw (1,-0.6) node{$a \mapsto b, \ b \mapsto a$};
\end{scope}
\end{tikzpicture}
\caption{$\rho$}
\label{}
\end{subfigure}
\caption{Graph automorphisms of Figure~\subref{graph5}}\label{fig:graph_auto_q2}
\end{figure}



\begin{proposition} \label{pro-aut-q2}
$\GG(\mathrm{Q2})=\SO(4)$.
\end{proposition}

\begin{proof}
We suppose that $\GG(\mathrm{Q2})$ properly contains $\SO(4)$ and will deduce a contradiction. Since the rank of $\GG(\textrm{Q2})$ is two, $\GG(\textrm{Q2})$ must be $G_2$ or $\Sp(2)$ or $\SO(5)$. However, $\GG(\textrm{Q2})\not=G_2$ because the maximal compact connected Lie subgroup of $G_2$ is isomorphic to either $\SO(4)$ or $\SU(3)$, but $\dim G_2/\SO(4)=8$ and $G_2/\SU(3)$ is diffeomorphic to $S^6$ while our manifold $M$ is diffeomorphic to $Q_3$.  We see that $\Sp(2)$ does not contain a Lie subgroup isomorphic to $\SO(4)$ although it contains a Lie subgroup isomorphic to $\Sp(1)\times \Sp(1)$. Thus, $\GG(\textrm{Q2})$ must be $\SO(5)$. 

As observed above, the restricted $\SO(4)$-action on $M$ has an orbit of dimension $5$, so the $\SO(5)$-action on $M$ has an orbit of dimension at least $5$ but at most $6$ because $\dim M=6$. The $\SO(5)$-orbit is isomorphic to $\SO(5)/H$, where $H$ is a compact Lie subgroup of $\SO(5)$ whose Lie algebra is of dimension $4$ or $5$ because $\dim \SO(5)=10$ and $5\le \dim (\SO(5)/H)\le 6$. One can easily check that the Lie subalgebra of $\mathfrak{so}(5)$ with dimension $4$ or $5$ is isomorphic to $\mathfrak{so}(3)\oplus \mathfrak{so}(2)$. Therefore, $\dim (\SO(5)/H)=6$. This means that the $\SO(5)$-action on $M$ is transitive because $\dim M=6$ and $M$ is connected. 

We note that $H$ is connected. Indeed, if we denote the identity component of $H$ by $H_0$, then we have a covering $\SO(5)/H_0\to \SO(5)/H=M$ with $H/H_0$ as the covering transformation. Since $M=Q_3$ is simply connected, $H/H_0$ must consist of a single element, which means that $H$ is connected. 

Since $H$ is connected and the Lie algebra of $H$ is isomorphic to $\mathfrak{so}(3)\oplus\mathfrak{so}(2)$, $H$ is isomorphic to either $\SO(3)\times \SO(2)$ or $\U(2)$. 
One can see that compact Lie subgroups of $\SO(5)$ isomorphic to $\U(2)$ are conjugate to each other in $\SO(5)$ by looking at Lie subalgebras of $\SO(5)$, and it is known that $\SO(5)/\U(2)$ is diffeomorphic to $\C P^3$. Therefore, $H$ must be isomorphic to $\SO(3)\times \SO(2)$. 

Compact Lie subgroups of $\SO(5)$ isomorphic to $\SO(3)\times\SO(2)$ are also conjugate to each other in $\SO(5)$. So we may assume that $H$ is $\SO(3)\times \SO(2)$ embedded in $\SO(5)$ in the standard way. Since our almost complex structure is invariant under the action of $\GG(\textrm{Q2})=\SO(5)$, it is determined by a complex structure on $\mathfrak{so}(5)/(\mathfrak{so}(3)\oplus\mathfrak{so}(2))$ invariant under the adjoint action of $\SO(3)\times \SO(2)$. Here, the real $\SO(3)\times \SO(2)$-module $\mathfrak{so}(5)/(\mathfrak{so}(3)\oplus\mathfrak{so}(2))$ is irreducible, so the invariant complex structure on it is unique up to sign and one can see that the associated GKM graph is the same as that of Type $B$. This is a contradiction and hence $\GG(\textrm{Q2})=\SO(4)$.
\end{proof}

\subsection{Type (S)}

Let $M$ be the manifold constructed in Subsection~\ref{sec:s}, which is diffeomorphic to $S^2 \times S^4$.
We can see that 
\begin{equation} \label{eq:aut_GM_D_2}
\Aut^*(\GM)=\begin{cases} \langle \tau\rangle \quad&\text{if $\{a,b\}\not=\{c,d\}$},\\
\langle \tau\rangle \times \langle \rho\rangle\quad&\text{if $\{a,b\}=\{c,d\}$},
\end{cases}
\end{equation}
where $\tau$ is the reflection through the middle vertical line of the GKM graph (as the graph automorphism) together with the mapping $\{a,b\}$ to $\{-c,-d\}$ and $\rho$ is the identity (as the graph automorphism) together with swapping $a$ and $b$. Both $\tau$ and $\rho$ are involutions (Figure~\ref{fig:graph_auto_D2}.) 

\begin{figure}[H]
\begin{subfigure}[b][3.7cm][s]{.4\textwidth}
\centering
\begin{tikzpicture}[state/.style ={circle, draw}]
\draw[->-=.5] (0,0) to (2,0);
\draw[->-=.5] (0,2) to [bend left=20]  (0,0);
\draw[->-=.5] (0,2) to [bend right=20]  (0,0);
\draw[->-=.5] (2,0) to [bend left=20]  (2,2);
\draw[->-=.5] (2,0) to [bend right=20]  (2,2);
\draw[->-=.5] (2,2) to (0,2);
\draw (1,-0.3) node{$a+b$};
\draw (-0.4,1) node{$a$};
\draw (0.5,1) node{$b$};
\draw (1.6,1) node{$c$};
\draw (2.4,1) node{$d$};
\draw (1,2.3) node{$a+b$};
\end{tikzpicture}
\caption{} \label{}
\end{subfigure} 
\begin{subfigure}[b][3.7cm][s]{.4\textwidth}
\centering
\begin{tikzpicture}[state/.style ={circle, draw}]
\draw[->-=.5] (0,0) to (2,0);
\draw[->-=.5] (0,2) to [bend left=20]  (0,0);
\draw[->-=.5] (0,2) to [bend right=20]  (0,0);
\draw[->-=.5] (2,0) to [bend left=20]  (2,2);
\draw[->-=.5] (2,0) to [bend right=20]  (2,2);
\draw[->-=.5] (2,2) to (0,2);
\draw (-0.7,1) node{$\xi:=$};
\draw (0,-0.3) node{$a+b$};
\draw (2,-0.3) node{$0$};
\draw (0,2.3) node{$a+b$};
\draw (2,2.3) node{$0$};
\draw (3,1) node{};
\end{tikzpicture}
\caption{} \label{}
\end{subfigure} 
\begin{subfigure}[b][4.4cm][s]{.32\textwidth}
\centering
\begin{tikzpicture}[state/.style ={circle, draw}]
\draw[->-=.4] (2,0) to (0,0);
\draw[->-=.5] (0,0) to [bend left=20]  (0,2);
\draw[->-=.5] (0,0) to [bend right=20]  (0,2);
\draw[->-=.5] (2,2) to [bend left=20]  (2,0);
\draw[->-=.5] (2,2) to [bend right=20]  (2,0);
\draw[->-=.4] (0,2) to (2,2);
\draw (1,-0.2) to (1,2.5);
\draw (1,2.3) node{$\longleftrightarrow$};

\draw (1,-0.6) node{$\{a,b\} \to \{-c,-d\}$};
\end{tikzpicture}
\caption{$\tau$}
\label{}
\end{subfigure}
\begin{subfigure}[b][4.4cm][s]{.3\textwidth}
\centering
\begin{tikzpicture}[state/.style ={circle, draw}]
\draw[->-=.5] (2,0) to (0,0);
\draw[->-=.4] (0,0) to [bend left=20]  (0,2);
\draw[->-=.4] (0,0) to [bend right=20]  (0,2);
\draw[->-=.4] (2,2) to [bend left=20]  (2,0);
\draw[->-=.4] (2,2) to [bend right=20]  (2,0);
\draw[->-=.5] (0,2) to (2,2);
\draw (-0.5,1) to (2.5,1);
\draw (-0.7,1) node{$\updownarrow$};
\draw (1,-0.4) node{};
\draw (1,2.4) node{};

\draw (1,-0.6) node{$\{a,b\} \to \{-a,-b\}$};
\draw (1,-0.9) node{$\xi \to -\xi$};
\end{tikzpicture}
\caption{}
\label{}
\end{subfigure} 
\begin{subfigure}[b][4.4cm][s]{.32\textwidth}
\centering
\begin{tikzpicture}[state/.style ={circle, draw}]
\draw[->-=.5] (0,0) to (2,0);
\draw[->-=.5] (0,2) to [bend left=30]  (0,0);
\draw[->-=.5] (0,2) to [bend right=30]  (0,0);
\draw[->-=.5] (2,0) to [bend left=30]  (2,2);
\draw[->-=.5] (2,0) to [bend right=30]  (2,2);
\draw[->-=.5] (2,2) to (0,2);
\draw (1,-0.4) node{};
\draw (1,2.4) node{};
\draw (0,1) node{$\leftrightarrow$};
\draw (2,1) node{$\leftrightarrow$};

\draw (1,-0.3) node{$a \leftrightarrow b$};
\draw (1,-0.7) node {$\{a,b\}=\{c,d\}$};
\end{tikzpicture}
\caption{$\rho$}
\label{}
\end{subfigure} 
\caption{Graph automorphisms}\label{fig:graph_auto_D2}
\end{figure}
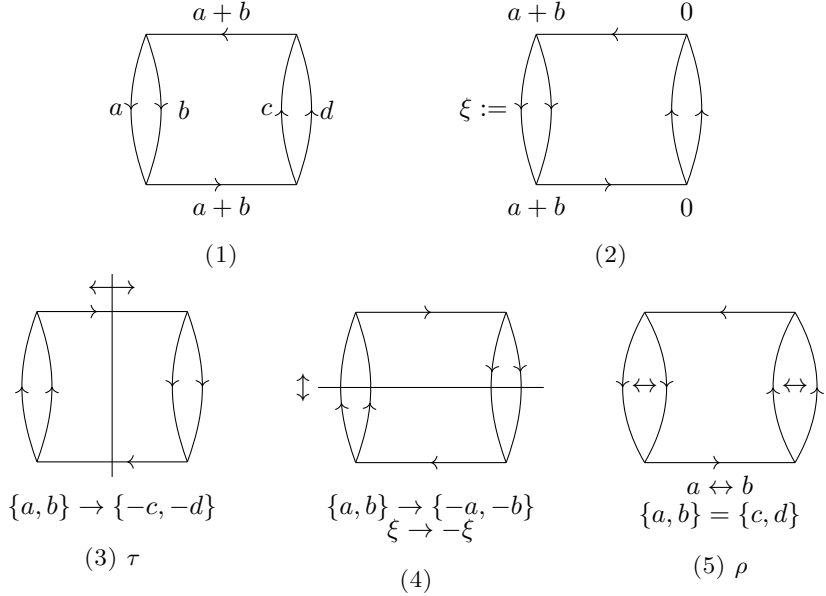

\begin{remark}
The reflection through the middle horizontal line of the GKM graph (as the graph automorphism) together with changing the sign of labels, the middle one in Figure~\ref{fig:graph_auto_D2}, is an element of $\Aut(\GM)$ but this acts nontrivially on $H^2(\GM)$, indeed it acts on $H^2(\GM)$ as multiplication by $-1$. 
\end{remark}

When $\{a,b\}=\{c,d\}$, we can see that a connected maximal compact Lie subgroup $\GG(\textrm{S1})$ of $\Aut(M)$ is $\GG(\textrm{S1})=\SO(4)$, see subsection~\ref{app:A2} in the appendix.

\appendix

\section{The octonions} \label{sec:octo}

The octonions are an $8$-dimensional algebra with basis $e_0=1, e_1,e_2,e_3,e_4,e_5,e_6,e_7$ such that their multiplication satisfies 
\begin{enumerate}
\item $e_k^2=-1$ for $1\le k\le 7$,
\item $e_ie_j=-e_je_i$ for $1\le i\not=j\le 7$, and 
\item $e_ie_j$ for $1\le i\not=j\le 7$ is either $e_k$ or $-e_k$ for some $1\le k(\not=i,j)\le 7$.
\end{enumerate} 
There are many choices of multiplications which satisfies the third condition (3) above. 
Usually the multiplication rule is given by a multiplication table or a Fano plane, see Figure~\ref{fig:1} below. It is known that the isomorphism type of $\O$ as algebra does not depend on the definition of multiplication. It is also known that the group of algebra automorphisms of $\O$, denoted by $\Aut(\O)$, is the exceptional compact Lie group $G_2$.

\begin{figure}[H]
\begin{subfigure}[b][2.7cm][s]{.6\textwidth}
\centering
\begin{tabular}{|c|c|c|c|c|c|c|c|}
\hline
 & $e_1$ & $e_2$ & $e_3$ & $e_4$ & $e_5$ & $e_6$ & $e_7$ \\ \hline
$e_1$ & $-1$ & $e_3$ & $-e_2$ & $e_5$ & $-e_4$ & $-e_7$ & $e_6$ \\ \hline
$e_2$ & $-e_3$ & $-1$ & $e_1$ & $e_6$ & $e_7$ & $-e_4$ & $-e_5$ \\ \hline
$e_3$ & $e_2$ & $-e_1$ & $-1$ & $e_7$ & $-e_6$ & $e_5$ & $-e_4$ \\ \hline
$e_4$ & $-e_5$ & $-e_6$ & $-e_7$ & $-1$ & $e_1$ & $e_2$ & $e_3$ \\ \hline
$e_5$ & $e_4$ & $-e_7$ & $e_6 $& $-e_1$ & $-1$ & $-e_3$ & $e_2$ \\ \hline
$e_6$ & $e_7$ & $e_4$ & $-e_5$ & $-e_2$ & $e_3$ & $-1$ & $-e_1$ \\ \hline
$e_7$ & $-e_6$ & $e_5$ & $e_4$ & $-e_3$ & $-e_2$ & $e_1$ & $-1$ \\ \hline
\end{tabular}
\end{subfigure} 
\begin{subfigure}[b][2.7cm][s]{.35\textwidth}
\centering
\begin{tikzpicture}[state/.style ={circle, draw}]
\draw[->-=.97] (0,1) circle (1);
\draw[->-=.3] (0,3) to (-1.73,0);
\draw[->-=.3] (0,3) to (1.73,0);
\draw[->-=.3] (1.73,0) to (-1.73,0);
\draw[->-=.2] (-1.73,0) to (0.86,1.5);
\draw[->-=.2] (-1.73,0) to (0.86,1.5);
\draw[->-=.2] (-0.86,1.5) to (1.73,0);
\draw[->-=.2] (0,3) to (0,0);
\draw (0,3.3) node{$e_1$};
\draw (-1.16,1.6) node{$e_2$};
\draw (-1.93,-0.2) node{$e_3$};
\draw (0.2,1.4) node{$e_4$};
\draw (0,-0.3) node{$e_5$};
\draw (1.93,-0.2) node{$e_6$};
\draw (1.16,1.6) node{$e_7$};
\end{tikzpicture}
\end{subfigure}
\caption{Multiplication table and Fano plane}\label{fig:1}
\end{figure}
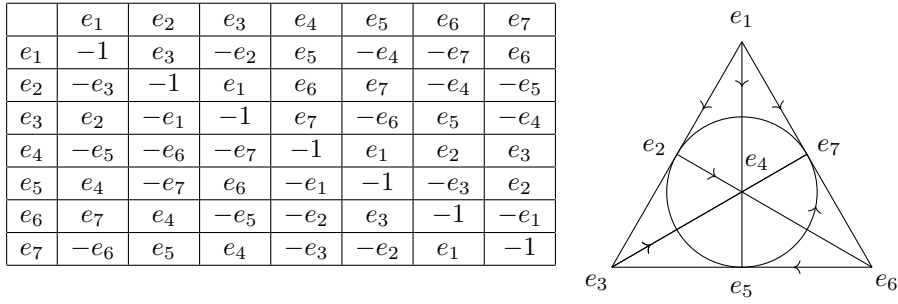

We consider pairs of quaternions with (natural addition and) multiplication defined by
\begin{equation} \label{eq:multiplication_1}
(p_1,q_1)(p_2,q_2)=(p_1 q_1-\overline{q_2} q_1, q_2 p_1+q_1 \overline{p_2})\qquad (p_1,q_1,p_2,q_2 \in \H)
\end{equation}
where $\bar{q}$ denotes the conjugate of $q\in \H$. The set of all pairs of quaternions with multiplication \eqref{eq:multiplication_1} is the octonions. Indeed, if we regard $\sum_{i=0}^7 x_ie_i$ $(x_i\in \R)$ as a pair of quaternions $(x_0+x_1i+x_2j+x_3k, x_4+x_5i+x_6j+x_7k)$, then the multiplication \eqref{eq:multiplication_1} agrees with the multiplication given by the multiplication table in Figure~\ref{fig:1}\footnote{In \cite{baez02}, the multiplication of pairs of quaternions is defined by 
$(p_1,q_1)(p_2,q_2)=(p_1 p_2- q_2 \overline{q_1}, \overline{p_1} q_2 + p_2 q_1)$ 
and this corresponds to a multiplication table different from Figure~\ref{fig:1}.}.

\begin{lemma} \label{lemm:A1}
We think of $\O$ as pairs of quaternions with multiplication \eqref{eq:multiplication_1}. Then, $\varphi_s$ and $\psi_t$ defined by 
\[
\varphi_s(p,q)=(sps^{-1},qs^{-1}) \quad (s\in \Sp(1)),\quad \psi_t(p,q)=(p,tq)\quad (t\in \Sp(1))
\]
are elements of $\Aut(\O)$ and they generate a subgroup of $\Aut(\O)$ isomorphic to $\SO(4)$. 
\end{lemma}

\begin{proof}
It is clear that the maps defined in the lemma are linear automorphisms of $\O$ and preserve multiplication \eqref{eq:multiplication_1}, so they are elements of $\Aut(\O)$. They commute with each other and the kernel of the homomorphism 
\[
\varphi\times \psi\colon \Sp(1)\times \Sp(1)\to \Aut(\O),\qquad (s,t)\to \varphi_s\psi_t
\]
is the group of order two generated by $(-I,-I)$. Since $(\Sp(1)\times \Sp(1))/\langle (-I,-I)\rangle$ is isomorphic to $\SO(4)$ as is well-known, the lemma follows. 
\end{proof}

\begin{remark} \label{rem:a1}
A direct computation shows that the subgroup $\SO(4)$ in Lemma~\ref{lemm:A1} leaves the unit sphere in the $3$-dimensional space spanned by $e_1, e_2, e_3$ invariant.
\end{remark}

\begin{lemma} \label{lemm:A1-1}
The isotropy subgroup at $(p,q)\in\O$ of the $\SO(4)$-action in Lemma~\ref{lemm:A1} is isomorphic to 
\[
\begin{cases} 
\SO(4)\quad&\text{if ${\rm Im}(p)=0$ and $q=0$},\\
\Sp(1)/\langle -I\rangle\cong\SO(3)\quad&\text{if ${\rm Im}(p)=0$ and $q\not=0$},\\
(\U(1)\times \Sp(1))/\langle (-I,-I)\rangle\quad&\text{if ${\rm Im}(p)\not=0$ and $q=0$},\\
\U(1) \quad& \text{if ${\rm Im}(p)\not=0$ and $q\not=0$},\\
\end{cases}
\]
where ${\rm Im}(p)$ denotes the imaginary part of $p\in \H$. 
\end{lemma}

\begin{proof}
Suppose that ${\rm Im}(p)=0$, that is, $p$ is a real number. Then $\varphi_s\psi_t(p,q)=(p,tqs^{-1})$.  
Therefore, $\varphi_s\psi_t(p,q)=(p,q)$ if and only if $q=tqs^{-1}$.  In particular, $(p,0)$ is fixed for any $s,t\in \Sp(1)$, proving the first case.  When $q\not=0$, $q=tqs^{-1}$ is equivalent to $t=qsq^{-1}$; so $t$ is determined by $s$.  Since the subgroup $\{(s, qsq^{-1})\mid s\in \Sp(1)\}$ contains the subgroup $\langle (-I,-I)\rangle$, the second case follows.  

As is well-known, the homomorphism 
\[
\rho\colon \Sp(1)\to \GL({\rm Im}\ \H)\quad\text{given by}\ \ s\to (x\to sxs^{-1})
\]
induces the standard representation of $\SO(3)\cong\Sp(1)/\{\pm 1\}$ on $\R^3={\rm Im}\ \H$, where ${\rm Im}\ \H$ denotes the imaginary part of $\H$ and $x\in {\rm Im}\ \H$. Therefore, when $x\not=0$, the image of the subgroup $H:=\{s\in \Sp(1)\mid sxs^{-1}=x\}$ by $\rho$ is a circle subgroup and this implies that $H$ is isomorphic to $\U(1)$, proving the third case.  The last case follows from this observation and the previous observation that $t=qsq^{-1}$ (so $t$ is determined by $s$) when $q\not=0$.   
\end{proof}

Since $G_2$ contains $\SU(3)$ as a subgroup, $\SU(3)$ should act on $\O$ as algebra automorphisms. However, it is unclear how it acts on $\O$ with multiplication \eqref{eq:multiplication_1}. In order to see that $\SU(3)$ acts on $\O$ as algebra automorphisms, the following multiplication is convenient (see \cite[Section 1.5]{Yo}). We think of an element of $\O$ as an element of $\C\oplus \C^3$ and denote it by 
\begin{equation} \label{eq:C+C3}
\text{$a+\mathbf{m}$\quad where $a\in \C$\ and\ $\mathbf{m}=\begin{pmatrix}m_1\\
m_2\\
m_3\end{pmatrix}\in \C^3$.}
\end{equation}
Then we consider the multiplication on $\O$ defined by 
\begin{equation} \label{eq:multiplication_3}
(a+\mathbf{m})(b+\mathbf{n})=(ab-\langle \mathbf{m},\mathbf{n}\rangle)+(a\mathbf{n}+\bar{b}\mathbf{m}-\overline{\mathbf{m}\times\mathbf{n}})
\end{equation}
where $\bar{\ }$ denotes the complex conjugation and 
\[
\mathbf{n}=\begin{pmatrix}n_1\\
n_2\\
n_3\end{pmatrix},\quad \langle \mathbf{m},\mathbf{n}\rangle=\sum_{i=1}^3m_i\bar{n}_i,\quad \mathbf{m}\times\mathbf{n}=\begin{pmatrix} m_2n_3-m_3n_2\\
m_3n_1-m_1n_3\\
m_1n_2-m_2n_1\end{pmatrix}.
\]

It is not difficult to see that the action of $A\in \SU(3)$ on $\O$ defined by 
\[
a+\mathbf{m}\to a+A\mathbf{m} 
\] 
is an algebra automorphism of $\O$ with multiplication \eqref{eq:multiplication_3} (see \cite[Theorem 1.9.1]{Yo}). This shows that $\SU(3)$ is a subgroup of $\Aut(\O)=G_2$.

\begin{remark}
(1) The multiplication \eqref{eq:multiplication_3} is given by the following multiplication table or the Fano plane. 
\begin{figure}[H]
\begin{subfigure}[b][2.7cm][s]{.6\textwidth}
\centering
\begin{tabular}{|c|c|c|c|c|c|c|c|}
\hline
 & $e_1$ & $e_2$ & $e_3$ & $e_4$ & $e_5$ & $e_6$ & $e_7$ \\ \hline
$e_1$ & $-1$ & $e_3$ & $-e_2$ & $e_5$ & $-e_4$ & $e_7$ & $-e_6$ \\ \hline
$e_2$ & $-e_3$ & $-1$ & $e_1$ & $-e_6$ & $e_7$ & $e_4$ & $-e_5$ \\ \hline
$e_3$ & $e_2$ & $-e_1$ & $-1$ & $e_7$ & $e_6$ & $-e_5$ & $-e_4$ \\ \hline
$e_4$ & $-e_5$ & $e_6$ & $-e_7$ & $-1$ & $e_1$ & $-e_2$ & $e_3$ \\ \hline
$e_5$ & $e_4$ & $-e_7$ & $-e_6 $& $-e_1$ & $-1$ & $e_3$ & $e_2$ \\ \hline
$e_6$ & $-e_7$ & $-e_4$ & $e_5$ & $e_2$ & $-e_3$ & $-1$ & $e_1$ \\ \hline
$e_7$ & $e_6$ & $e_5$ & $e_4$ & $-e_3$ & $-e_2$ & $-e_1$ & $-1$ \\ \hline
\end{tabular}
\end{subfigure} 
\begin{subfigure}[b][2.7cm][s]{.35\textwidth}
\centering
\begin{tikzpicture}[state/.style ={circle, draw}]
\draw[->-=.97] (0,1) circle (1);
\draw[->-=.3] (0,3) to (-1.73,0);
\draw[->-=.3] (1.73,0) to (0,3);
\draw[->-=.3] (-1.73,0) to (1.73,0);
\draw[->-=.2] (-1.73,0) to (0.86,1.5);
\draw[->-=.2] (-1.73,0) to (0.86,1.5);
\draw[->-=.2] (1.73,0) to (-0.86,1.5);
\draw[->-=.2] (0,3) to (0,0);
\draw (0,3.3) node{$e_1$};
\draw (-1.16,1.6) node{$e_2$};
\draw (-1.93,-0.2) node{$e_3$};
\draw (0.2,1.4) node{$e_4$};
\draw (0,-0.3) node{$e_5$};
\draw (1.93,-0.2) node{$e_6$};
\draw (1.16,1.6) node{$e_7$};
\end{tikzpicture}
\end{subfigure}
\caption{Multiplication table and Fano plane corresponding to \eqref{eq:multiplication_3}}\label{fig:2}
\end{figure}
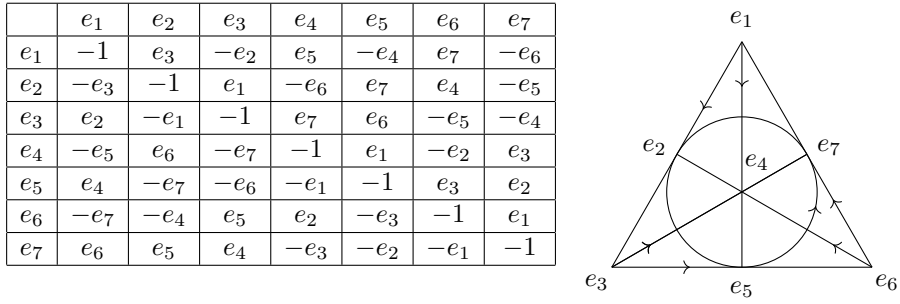

(2) If we write the multiplication \eqref{eq:multiplication_1} using the expression \eqref{eq:C+C3}, then 
\[
(a+\mathbf{m})(b+\mathbf{n})=(ab-m_1\bar{n}_1-m_2\bar{n}_2-\bar{m}_3n_3)+\begin{pmatrix} an_1+\bar{b}m_1+\bar{m}_2n_3-m_3\bar{n}_2\\
an_2-\bar{b}m_2+m_3\bar{n}_1-\bar{m}_1n_3\\
\bar{a}n_3+bm_3+{m}_1n_2-m_2n_1
\end{pmatrix}.
\]
Comparing this with \eqref{eq:multiplication_3}, we see that the map defined by 
\begin{equation} \label{eq:iso_13}
a+\begin{pmatrix}m_1\\
m_2\\
m_3\end{pmatrix}\to a+\begin{pmatrix}m_1\\
m_2\\
-\bar{m}_3\end{pmatrix}
\end{equation}
gives an algebra isomorphism of $\O$ with multiplication \eqref{eq:multiplication_1} to $\O$ with multiplication \eqref{eq:multiplication_3}. 
\end{remark}

\subsection{Almost complex structure on $S^6$}
We denote the real part of $u\in \O$ by $\Real u$ and set 
\[
\Imag\O=\{u\in \O\mid \Real u=0\}.
\]
The octonions $\O$ has the standard inner product
\[
(u,v)=u_0v_0+\sum_{i=1}^7 u_iv_i \qquad \text{for $u=u_0+\sum_{i=1}^7u_ie_i$ and $v=v_0+\sum_{i=1}^7v_ie_i$}.
\] 
Therefore,
\begin{equation} \label{eq:inner_product}
(u,v)=-\Real(uv)\quad \text{if}\ u,v\in \Imag\O. 
\end{equation} 

We define 
\[
S^6=\{ u\in \Imag\O\mid |u|=1\}
\]
where $|u|=\sqrt{(u,u)}$. Noting \eqref{eq:inner_product}, we can describe the total space of the tangent bundle of $S^6$ as
\[
TS^6=\{(u,v)\in \Imag\O\times \Imag\O\mid |u|=1,\ \Real(uv)=0\}. 
\] 
The almost complex structure $J_u$ on the tangent space $T_uS^6$ of $S^6$ at $u\in S^6$ is defined by $J_u(u,v)=(u,uv)$. The action of $\Aut(\O)=G_2$ on $\O$ leaves $S^6$ invariant and preserves the almost complex structure $J=\{J_u\}_{u\in S^6}$ on $S^6$. The action of $\Aut(\O)=G_2$ is transitive and the isotropy subgroup at $\pm e_1$ (the north and south poles of $S^6$) is isomorphic to $\SU(3)$.

\begin{example} \label{exa:s6}
Using the multiplication \eqref{eq:multiplication_3}, let $T=T^2$ act on $S^6$ by
\[
g \cdot \left( a+\begin{pmatrix}m_1\\
m_2\\
m_3\end{pmatrix} \right) =a+\begin{pmatrix} g^{-\alpha-\beta} & 0 & 0 \\ 0 & g^{\alpha} & 0 \\ 0 & 0 & g^{\beta} \end{pmatrix}  \begin{pmatrix}m_1\\
m_2\\
m_3\end{pmatrix}
\]
for all $a+\mathbf{m} \in S^6$ and $g \in T$, for some $\alpha,\beta$ that form a basis of $H^2(BT)$.
This action has two fixed points $e_1$ and $-e_1$ that have weights $\{-\alpha-\beta,\alpha,\beta\}$ and $\{\alpha+\beta,-\alpha,-\beta\}$, respectively. This is a GKM manifold with GKM graph Figure~\ref{s6}.
\end{example}

\subsection{Almost complex structure on $S^2\times S^4$} \label{app:A2}
We define
\[
S^2\times S^4=\{u\in \O\mid u_1^2+u_2^2+u_3^2=1=u_0^2+u_4^2+u_5^2+u_6^2+u_7^2\}. 
\]
For $u\in S^2\times S^4$, we set 
\[
n_1=\sum_{i=1}^3u_ie_i,\quad n_2=u_0+\sum_{i=4}^7u_ie_i.
\]
Then the tangent space $T_u(S^2\times S^4)$ is given by 
\[
T_u(S^2\times S^4)=\{(u,v)\in (S^2\times S^4)\times \O\mid \langle n_1,v\rangle=0=\langle n_2,v\rangle\}.
\]

\begin{lemma}
Let $u,v, n_1,n_2$ be as above. Then $\langle n_1, (n_2n_1)v\rangle=0=\langle n_2,(n_2n_1)v\rangle$. 
\end{lemma}

\begin{proof}
First we prove the former identity.  We note that 
\[
\begin{split}
n_2n_1&=(u_0+\sum_{i=4}^7u_ie_i)(\sum_{i=1}^3u_ie_i)=u_0(\sum_{i=1}^3u_ie_i)+(\sum_{i=4}^7u_ie_i)(\sum_{i=1}^3u_ie_i)\\
&=(\sum_{i=1}^3u_ie_i)u_0-(\sum_{i=1}^3u_ie_i)(\sum_{i=4}^7u_ie_i)=n_1\bar{n}_2.
\end{split}
\]
Therefore 
\[
\overline{(n_2n_1)}n_1=\overline{(n_1\bar{n}_2)}n_1=({n}_2\bar{n}_1)n_1={n}_2(\bar{n}_1n_1)=n_2.
\]
Hence 
\[
\langle n_1,(n_2n_1)v\rangle=\langle \overline{(n_2n_1)}n_1,v\rangle=\langle {n}_2,v\rangle=0. 
\]
This proves the first identity. 
As for the second identity, we note
\[
\overline{(n_2n_1)}n_2=(\bar{n}_1\bar{n}_2)n_2=\bar{n}_1(\bar{n}_2n_2)=\bar{n}_1=-n_1.
\]
Hence
\[
\langle n_2,(n_2n_1)v\rangle=\langle\overline{(n_2n_1)}n_2,v\rangle=-\langle n_1,v\rangle=0.
\]
This proves the second identity. 
\end{proof}

We define $J_u\colon T_u(S^2\times S^4)\to T_u(S^2\times S^4)$ by 
\[
J_u(u,v)=(u,(n_2n_1)v).
\]
Since $n_2n_1\in \Imag\O$, we have $(n_2n_1)^2=-|n_2|^2|n_1|^2=-1$.  Therefore, $\{J_u\}_{u\in S^2\times S^4}$ defines an almost complex structure on $S^2\times S^4$. This verifies the almost complex structure on $S^{2}\times S^{4}$ in \cite{DS}.

It is not difficult to see that the subgroup of $\Aut(\O)$ generated by $\varphi_p$ and $\psi_q$ in Lemma~\ref{lemm:A1}, which is isomorphic to $\SO(4)$, preserves $S^2\times S^4$ with the almost complex structure $\{J_u\}_{u\in S^2\times S^4}$ above. 

\section{Deformation of almost complex structures} \label{sec:deform}

Let $G$ be a compact Lie group and $(X,J^X)$ an almost complex $G$-manifold with a $G$-fixed point $p$. 
We show that a neighborhood of $p$ may be thought of as a complex $G$-module (locally) by deforming the given almost complex structure $J^X$ continuously. This enables us to blow up $p$ in $X$ equivariantly, producing an almost complex $G$-manifold diffeomorphic to $X\# \overline{\C P^n}$ where $n=\dim_\C X$. 

The argument is as follows. We take a $G$-equivariant smooth map 
\[
f\colon T_pX\to X
\]
which is a diffeomorphism on a $G$-invariant open disk $W$ (with $0$ centered) in $T_pX$. 
One way to get such $f$ is to choose a $G$-invariant Riemannian metric on (a $G$-invariant neighborhood of $p$ in) $X$ and take an associated exponential map at $p$. Through $f$, we obtain a $G$-invariant almost complex structure $J=\{J_z\}_{z\in W}$ on $W$, where $J_z$ is the induced complex structure on $T_z(T_pX)(=T_pX)$, i.e. $J_z:=df^{-1}\circ J_{f(z)}^X\circ df$. Note that $J_0=J_p^X$ since $f(0)=p$ and $df$ is the identity at $p$. Thus, we are led to study a continuous deformation of a $G$-invariant almost complex structure on $W(\subset T_pX)$ as a manifold. 

In general, let $V$ be a real $G$-module with a $G$-invariant metric and $W$ a $G$-invariant open disk (with $0$ centered) in $V$. We take $V=T_pX$ in our application. We think of $W$ as a $G$-manifold and assume that it has a $G$-invariant almost complex structure $J$, namely $J=\{J_z\}_{z\in W}$ where $J_z$ is a complex structure on $T_zW(=V)$ such that 
\begin{equation} \label{eq:G_invariant_almost_complex_structure}
dg\circ J_z=J_{gz}\circ dg\quad\text{for any $z\in W$ and $g\in G$},
\end{equation} 
where $dg\colon TW\to TW$ denotes the differential of $g\colon W\to W$. 

\begin{lemma} \label{lemm:deformation}
There exists a $G$-invariant almost complex structure $\MJ$ on $W$ such that $\MJ$ is homotopic to $J$ and 
\[
\MJ_z=\begin{cases} J_0 \quad &\text{for $z$ with $|z|<\epsilon_1$}\\
J_z\quad &\text{for $z$ with $|z|>\epsilon_2$}
\end{cases}
\]
where $0<\epsilon_1<\epsilon_2$ and $\{z\in V\mid |z|<\epsilon_2\}\subset W$. 
\end{lemma}

\begin{proof}
Take a continuous homotopy $\rho_t\colon W\to [0,1]$ $(t\in [0,1])$ of continuous $G$-invariant functions such that 
\begin{equation} \label{eq:rho}
\begin{split}
\rho_0(z)&=1\quad \text{for $\forall z\in W$},\\
\rho_t(z) &=1\quad \text{for $\forall t\in [0,1]$ and $\forall z$ with $|z|>\epsilon_2$},\\
\rho_1(z)&=0\quad \text{for $\forall z$ with $|z|<\epsilon_1$}.
\end{split}
\end{equation}
It is not difficult to find such $\rho_t$. For instance, take a continuous function $\gamma\colon [0,\infty)\to [0,1]$ such that 
\[
\gamma(x)=\begin{cases} 0 \quad& \text{for $x$ with $x<\epsilon_1$}\\
1\quad&\text{for $x$ with $x>\epsilon_2$}\end{cases}
\]
and set $\gamma_t(x)=t\gamma(x)+(1-t)$ for $t\in [0,1]$. Then $\rho_t(z):=\gamma_t(|z|)$ is a desired one. 

For each $t\in [0,1]$, we define an almost complex structure on $T_zW$ by 
\begin{equation} \label{eq:deformation}
\MJ^t_z:=J_{\rho_t(z)z}.
\end{equation}
(Note that $\rho_t(z)z\in W$ whenever $z\in W$ because $\rho_t(z)\in [0,1]$ and $W$ is an open disk centered at $0$.) The almost complex structure $\{\MJ^t_z\}_{z\in W}$ on $W$ is $G$-invariant for each $t\in [0,1]$ because 
\[
\begin{split}
dg\circ \MJ^t_z&=dg\circ J_{\rho_t(z)z}=J_{g(\rho_t(z)z)}\circ dg\\
&=J_{\rho_t(z)gz}\circ dg=J_{\rho_t(gz)gz}\circ dg=\MJ^t_{gz}\circ dg
\end{split}
\]
for any $z\in W$ and $g\in G$, where the second identity follows from \eqref{eq:G_invariant_almost_complex_structure}, the third from the linearity of the $G$-action on $V$, and the fourth from the $G$-invariance of $\rho$. 
It follows from \eqref{eq:rho} and \eqref{eq:deformation} that $\MJ^0_z=J_z$ for $\forall z\in W$, that is $\MJ^0=J$, and that 
\[
\MJ^1_z=\begin{cases} J_0 \quad&\text{for $z$ with $|z|<\epsilon_1$},\\
J_z\quad&\text{for $z$ with $|z|>\epsilon_2$}.\end{cases}
\]
Therefore, $\MJ^1$ is the desired $\MJ$. 
\end{proof}

The argument developed above works for blow-up of $X$ along a closed almost complex $G$-submanifold $Y$ of $X$. Indeed, let $V$ be a real $G$-vector bundle (with a $G$-invariant metric) over $Y$ and $W$ a $G$-invariant open disk bundle of $V$. We think of $W$ as a smooth $G$-manifold and assume that $W$ has a $G$-invariant almost complex structure $J=\{J_z\}_{z\in W}$, where $J_z$ is a complex structure on $T_zW$ as before. 
Then Lemma~\ref{lemm:deformation} can be generalized as follows by the same argument. 

\begin{lemma} \label{lemm:deformation2}
Let $\pi\colon W\to Y$ be the projection. 
There exists a $G$-invariant almost complex structure $\MJ$ on $W$ such that $\MJ$ is homotopic to $J$ and 
\[
\MJ_z=\begin{cases} J_{\pi(z)} \quad &\text{for $z$ with $|z|<\epsilon_1$}\\
J_z\quad &\text{for $z$ with $|z|>\epsilon_2$}
\end{cases}
\]
where $0<\epsilon_1<\epsilon_2$ and $\{z\in V\mid |z|<\epsilon_2\}\subset W$. 
\end{lemma}

We take a $G$-invariant Riemannian metric on $X$ and $V$ to be a normal $G$-vector bundle of $Y$ in $X$. Then $V$ has a $G$-invariant metric and the exponential map $f\colon V\to X$ is a $G$-equivariant smooth map which is a diffeomorphism on some $G$-invariant open disk bundle $W$ of $V$. Through $f$, we think of $W$ as a $G$-invariant open tubular neighborhood of $Y$ in $X$. Then Lemma~\ref{lemm:deformation2} enables us to blow up $X$ along $Y$, which we shall discuss in more detail below. 

Let $\W2$ be the open $\epsilon_2$-disk bundle of $V$, where $\epsilon_2$ is the positive real number in Lemma~\ref{lemm:deformation2}. $\W2$ is contained in $W$ and $G$-invariant. By Lemma~\ref{lemm:deformation2}, the fiber of $\W2\to Y$ over $y\in Y$ is an open disk of a complex vector space (with the complex structure $J_y$), so we may think of $\W2 \to Y$ as an open disk bundle of some complex $G$-vector bundle over $Y$, say $E\to Y$ whose realification is $V\to Y$. Thus, we obtain a complex projective bundle $P(E)\to Y$ associated to $E\to Y$. Let $L\to P(E)$ be a tautological line bundle, where 
\begin{equation} \label{eq:BW}
L:=\{(\ell, v)\mid \ell\in P(E),\ v\in\ell\}.
\end{equation}
We consider its open $\epsilon_2$-disk bundle $D(L)$, where $P(E)$ is embedded as the zero section of $D(L)$. If $(\ell,v)\in D(L)$ is not in $P(E)$, i.e. $v\not=0$, then the line $\ell$ is determined by $v$ and $v$ may be regarded as a point of $\W2$. Indeed, the projection $(\ell,v)\to v$ gives an almost complex $G$-equivariant diffeomorphism 
\[
\varphi\colon D(L)\backslash P(E)\to \W2\backslash Y. 
\]
Gluing $D(L)$ to $X\backslash Y$ via $\varphi$ produces an almost complex $G$-manifold $\widetilde{X}$ and this is the blow-up of $X$ along $Y$.

\begin{remark}
Set theoretically, $\widetilde{X}$ is obtained from $X$ by replacing $Y$ with $P(E)$. Indeed, $P(E)$ is embedded in $\widetilde{X}$ as a $G$-invariant closed submanifold. 
Let $q\colon P(E)\to Y$ be the projection. For $\ell\in P(E)\subset \widetilde{X}$, there is a natural isomorphism 
\[
T_\ell\widetilde{X}\cong {dq}_\ell^*(T_{q(\ell)}Y)\oplus T_\ell(P(E_{q(\ell)}))\oplus L_\ell,
\]
where $L_\ell$ denotes the fiber of $L\to P(E)$ over $\ell$. 
The almost complex structure on $T_\ell{\widetilde{X}}$ is given by the almost complex structure on the right hand side above. 
\end{remark}

\section{Weights of faithful $T$-modules} \label{sec:faithful}

Let $T$ be a compact torus of dimension $n$. Recall that for $u\in H^2(BT)$, we denote by $\chi^u$ the homomorphism from $T$ to $S^1$ through the natural identification $H^2(BT)\cong \Hom(T,S^1)$.  Let   $u_1,\dots,u_m\in H^2(BT)$ and we consider the homomorphism $\Phi\colon T\to T^m$ defined by
\begin{equation} \label{eq:Phi_A}
\Phi(g):=(\chi^{u_1}(g),\dots,\chi^{u_m}(g)).
\end{equation}
The map $\Phi\colon T\to T^m$ induces a homomorphism from $H^2(BT^m)$ to $H^2(BT)$, also from $\Hom(T^m,S^1)$ to $\Hom(T,S^1)$.  We often identify $H^2(BT^m)$ (resp. $H^2(BT)$) with $\Hom(T^m,S^1)$ (resp. $H^2(BT^m)$) and denote both the induced homomorphisms by $\Phi^*$.   

\begin{lemma} \label{lemm:Phi_A}
The image of $\Phi^*\colon H^2(BT^m)\to H^2(BT)$ is the subgroup of $H^2(BT)$ generated by $u_1,\dots,u_m$. 
\end{lemma}

\begin{proof}
Let $p_i\colon T^m\to S^1$ be the $i$-th projection $(i=1,\dots,m)$.  Then $\Phi^*({p_i})=p_i\circ \Phi=\chi^{u_i}$.  This implies the lemma because $H^2(BT^m)$ is generated by the weights of $p_i$'s. 
\end{proof}

\begin{proposition}
The homomorphism $\Phi\colon T\to T^m$ in \eqref{eq:Phi_A} is injective if and only if $u_1,\dots,u_m$ generate $H^2(BT)$. 
\end{proposition}

\begin{proof}
Suppose that $\Phi$ is injective. Then $m\ge n$ and $\Phi\colon T\to \Phi(T)$ is an isomorphism.  Since $\Phi(T)$ is an $n$-dimensional subtorus of $T^m$, there is a complementary $(m-n)$-dimensional subtorus $K$ of $T^m$ such that $\Phi(T)$ and $K$ generate $T^m$ and $\Phi(T)\cap K$ is the identity.  Therefore, there exists $p\in \Hom(T^m,T)$ such that $p\circ \Phi$ is the identity.  Indeed, the $p$ is the composition of the projection $T^m\to \Phi(T)$ and the isomorphism $\Phi^{-1}\colon \Phi(T)\to T$.  
Given an $f\in \Hom(T,S^1)$, we define $\hat{f}\in \Hom(T^m,S^1)$ to be $f\circ p$.  Then $\Phi^*(\hat{f})=\hat{f}\circ\Phi=f\circ p\circ \Phi=f.$
This shows that $\Phi^*$ is surjective, so $u_1,\dots, u_m$ generate $H^2(BT)$ by Lemma~\ref{lemm:Phi_A}. 

Conversely, suppose that $u_1,\dots,u_m$ generate $H^2(BT)$.  Then by Lemma~\ref{lemm:Phi_A} the map $\Phi^*\colon \Hom(T^m,S^1)\to \Hom(T,S^1)$ is surjective. Therefore, for any $f\in \Hom(T,S^1)$, there is a lift $\tilde{f}\in \Hom(T^m,S^1)$ such that $f=\tilde{f}\circ\Phi$.  This means that $\ker \Phi$ is contained in $\ker f$ for any $f\in \Hom(T,S^1)$. Since the intersection of $\ker f$ over all $f\in \Hom(T,S^1)$ is the identity, $\ker\Phi$ must be the identity, proving the injectivity of $\Phi$. 
\end{proof}


\end{document}